\documentclass[journal,twoside,web]{ieeecolor}
\usepackage{generic}

\makeatletter
\newcommand{\removelatexerror}{\let\@latex@error\@gobble}
\makeatother
\usepackage[T1]{fontenc}
\usepackage[utf8]{inputenc}
\usepackage{cite}
\usepackage{dsfont}
\usepackage[super]{nth}
\usepackage{paralist}
\let\labelindent\relax
\usepackage{enumitem}
\usepackage{bbm}

\newcommand{\xvbox}[2]{\makebox[#1][l]{#2}} 
\newcommand{\Pis}[1]{\Pi_{\mathrm{state}}(#1)}
\newcommand{\Pit}[1]{\Pi_{\mathrm{traj}}(#1)}
\newcommand{\Pic}[1]{\Pi_{\mathrm{cost}}(#1)}
\newcommand{\uxone}{u_x^{\mathrm{one}}}

\usepackage[english]{babel}
\usepackage[usenames,dvipsnames,table]{xcolor}
\usepackage{amsmath,amsthm,amssymb}
\usepackage{graphicx}
\usepackage{url}
\usepackage{float}

\usepackage{tikz}
\usetikzlibrary{shapes,arrows,positioning,automata}

\usepackage[mathscr]{euscript}
\usepackage{mathtools}

\usepackage{xargs} %
\usepackage{subcaption} %
\usepackage[font=footnotesize,labelfont=footnotesize]{caption}
\usepackage[noend,ruled,shortend,commentsnumbered,linesnumbered]{algorithm2e}
\usepackage{etoolbox} \let\bbordermatrix\bordermatrix
\patchcmd{\bbordermatrix}{8.75}{4.75}{}{}

\newcommand{\real}{\mathbb{R}}
\newcommand{\realnonnegative}{{\mathbb{R}}_{\ge 0}}

\newcommand{\naturalnumbers}{\mathbb{N}}
\newcommand{\norm}[1]{\ensuremath{\| #1 \|}}
\newcommand{\until}[1]{[\, #1 \,]}
\newcommand{\map}[3]{#1:#2 \rightarrow #3}
\newcommand{\setdef}[2]{\{#1 \; | \; #2\}}

\newcommand{\setmap}[3]{#1:#2 \rightrightarrows #3}
\newcommand{\ones}{\mathbf{1}}

\newcommand{\myemphc}[1]{\emph{#1}} 
\newcommand{\setr}[1]{\{#1\}}

\newcommand{\abs}[1]{|#1|}

\newcommand{\xtraj}{\mathsf{x}}
\newcommand{\utraj}{\mathsf{u}}
\newcommand{\Unsafe}{\mathcal{U}\mathcal{I}}
\newcommand{\JJ}{\mathcal{J}}
\newcommand{\Qo}{\overline{Q}}
\newcommand{\DDo}{\overline{\DD}}
\newcommand{\drmpc}{\mathtt{DR\_MPC}}
\newcommand{\rf}{\mathfrak{r}}
\newcommand{\st}{\operatorname{subject \text{$\, \,$} to}}
\renewcommand{\st}{\operatorname{s.t.}}

\newcommand{\Eb}{\mathbb{E}}
\newcommand{\Pb}{\mathbb{P}}

\newcommand{\Db}{\mathbb{D}}
\newcommand{\Wb}{\mathbb{W}}
\newcommand{\Data}{\widehat{\mathcal{W}}}

\newcommand{\DD}{\mathcal{D}}

\newcommand{\HH}{\mathcal{H}}
\newcommand{\II}{\mathcal{I}}
\newcommand{\KK}{\mathcal{K}}

\newcommand{\NN}{\mathcal{N}}
\newcommand{\OO}{\mathcal{O}}

\newcommand{\PP}{\mathcal{P}}
\newcommand{\QQ}{\mathcal{Q}}

\newcommand{\SSs}{\mathcal{S}}
\newcommand{\SSo}{\overline{\SSs}}

\newcommand{\UU}{\mathcal{U}}

\newcommand{\WW}{\mathcal{W}}
\newcommand{\XX}{\mathcal{X}}

\newcommand{\CVaR}{\operatorname{CVaR}}
\newcommand{\VaR}{\operatorname{VaR}}

\newcommand{\dist}{\operatorname{dist}}
\newcommand{\eps}{\epsilon}

\newcommand{\diam}{\operatorname{diam}}

\newcommand{\what}{\widehat{w}}

\newcommand{\costgo}[2]{J_{(#1:#2)}}

\newcommand{\jth}{j^{\mathrm{th}}}
	\newcommand{\xt}{\tilde{x}}
\newcommand{\ut}{\tilde{u}}

\newcommand{\oprocendsymbol}{\hbox{$\bullet$}}
\newcommand{\oprocend}{\relax\ifmmode\else\unskip\hfill\fi\oprocendsymbol}
\newcommand{\longthmtitle}[1]{\mbox{}\textup{\textsl{(#1):}}}

\allowdisplaybreaks

\newcommand{\ifinclude}[1]{}
\newcommand{\hide}[1]{}
\renewcommand{\hide}[1]{#1}

\renewcommand{\theenumi}{(\roman{enumi})}

\makeatletter
\newcommand{\thickhline}{%
  \noalign {\ifnum 0=`}\fi \hrule height 1pt
  \futurelet \reserved@a \@xhline
}
\newcolumntype{"}{@{\hskip\tabcolsep\vrule width 1pt\hskip\tabcolsep}}
\makeatother

\newtheorem{theorem}{Theorem}[section]
\newtheorem{proposition}{Proposition}[section]
\newtheorem{lemma}[theorem]{Lemma}

\theoremstyle{definition}

\newtheorem{assumption}{Assumption}[section]
\newtheorem{remark}[theorem]{Remark}

\usepackage[normalem]{ulem} %
\definecolor{new}{rgb}{0.55,0,0.55}
\usepackage[prependcaption,colorinlistoftodos]{todonotes}

\title{Iterative risk-constrained model predictive control: A data-driven  distributionally robust approach}
\author{Alireza Zolanvari \qquad Ashish Cherukuri \thanks{The authors are with the Engineering and Technology Institute Groningen, University of Groningen. Email: \texttt{\{a.zolanvari,a.k.cherukuri\}@rug.nl}. This work was partly supported with a scholarship from the Data Science and Systems Complexity (DSSC) Center, University of Groningen. }}

\begin{document}

\maketitle
\thispagestyle{empty}

\begin{abstract}
This paper proposes an iterative distributionally robust model predictive control (MPC) scheme to solve a  risk-constrained infinite-horizon optimal control problem. In each iteration, the algorithm generates a trajectory from the starting point to the target equilibrium state with the aim of respecting risk constraints with high probability (that encodes safe operation of the system) and improving the cost of the trajectory as compared to previous iterations. At the end of each iteration, the visited states and observed samples of the uncertainty are stored and accumulated with the previous observations. For each iteration, the states stored previously are considered as terminal constraints of the MPC scheme, and samples obtained thus far are used to construct distributionally robust risk constraints. As iterations progress, more data is obtained and the environment is explored progressively to ensure better safety and cost optimality. 
We prove that the MPC scheme in each iteration is recursively feasible and the resulting trajectories converge asymptotically to the target while ensuring safety with high probability. 
 We identify conditions under which the cost-to-go reduces as iterations progress. For systems with locally one-step reachable target, we specify scenarios that ensure finite-time convergence of iterations. We provide computationally tractable reformulations of the risk constraints for total variation and Wasserstein distance-based ambiguity sets.  A simulation example illustrates the application of our results in finding a risk-constrained path for two mobile robots facing an uncertain obstacle. 	
\end{abstract}

\section{Introduction}\label{sec:intro}

\IEEEPARstart{P}{ractical} control systems often operate in uncertain environments, for example, a mobile robot navigating in the presence of obstacles. Safe optimal control in such situations can be modeled using different sets of constraints. On the one hand, robust constraints  consider the worst-case effect of uncertainty on control design. On the other hand, popular probabilistic approaches embed chance constraints in the optimal control problem to ensure safety. A convenient strategy to balance these approaches is to consider risk constraints. We adopt this approach in our work and define a risk-constrained infinite-horizon optimal control problem. We assume that the task needs to be performed in an iterative way and the data regarding the uncertainty is incrementally revealed as iterations progress. For this setting, we design an iterative method that combines the notions of learning model predictive control~\cite{UR-FB:17-tac} and distributionally robust risk constraints~\cite{ARH-AC-JL:19-acc}.

\subsubsection*{Literature review}

	To solve optimization problems with uncertainty, and in particular risk constraints, one needs to know the underlying distribution of the random variable. Often this information is available to the decision-maker in terms of the samples (data). In the case, the dataset is small or samples are corrupted, data-driven distributionally robust (DR) optimization~\cite{HR-SM:19-arXiv} presents itself as an appealing solution. This DR framework solves optimization problems by taking a worst-case approach over a set of probability distributions termed ambiguity set. This set is constructed using samples, often in a manner that ensures a required level out-of-sample performance, which translates to safety under unknown scenarios in control problems. Taking advantage of this property of the DR framework, several recent works~\cite{ZZ-JJZ:23,  LA-MF-JL-FD:23, MF-JL:22, FM-TS-JL:22, CM-SL:21, JC-JL-FD:22, PC-PP:21, MS-PP:23, ZZ-EADR-PP:2023,AD-MA-JWB:2022-lcss} explore distributional robustness in the context of Model Predictive Control (MPC). 
	These can be broadly categorized into methods that focus on propagation of ambiguity sets~\cite{ZZ-JJZ:23, LA-MF-JL-FD:23}, tube-based MPC~\cite{MF-JL:22,CM-SL:21},  reformulations and finite-sample guarantees~\cite{AD-MA-JWB:2022-lcss, FM-TS-JL:22}, and chance-constraints and data-driven control~\cite{JC-JL-FD:22}. In~\cite{PC-PP:21, MS-PP:23}, reformulations and guarantees for DR-based MPC are explored while using the connection that the so-called coherent risk measure of a random variable is equivalent to the worst-case expectation over a set of distributions. Similar connection is used to investigate risk-averse MPC for Markovian switched systems in~\cite{PS-DH-AB-PP:18, SS-YC-AM-MP:19-tac}

Most of the above-listed works on MPC are focused on stochastic systems, while we consider deterministic systems that operate under uncertain environments.  As shown in~\cite{AH-GCK-IY:19, AD-MA-JWB:20, SXW-AD-ShT-JWB:2022}, this latter setup finds application in motion planning by encoding safety against collision in form of risk constraints. In~\cite{AH-IY:20-ICRA}, motion planning problem with risk constraints is considered where Wasserstein ambiguity sets are used to impose the risk constraints in  a distributionally robust manner. This approach was further extended in~\cite{AH-IY:23} where obstacles are assumed to have dynamics that is unknown to the planner and DR constraints are used to robustify against this lack of information. In~\cite{AN-ARH:2022}, DR risk-constrined MPC is studied in the case of multi-robot systems. In contrast with these approaches, we explore the possibility of executing the task in an iterative manner. An iterative approach is a natural one when the data regarding uncertainty is less initially, and more samples become available over time. As a consequence, the environment can be explored progressively. Such a framework also facilitates motion planning and control for the case where a reference trajectory is not available.  The safety level ensured in this process can be tuned using DR constraints. To realize such a method, we employ the learning model predictive framework introduced in~\cite{UR-FB:17-tac}. Here, at each iteration, a part of the state space is explored and stored for future iterations where these states are used as terminal constraints. In \cite{MB-CV-FB:20}, a learning-based MPC has been developed to tackle the uncertainties in the problem's constraints in a safe procedure. However, these strategies aim at satisfying robust and not risk constraints.

\subsubsection*{Setup and contributions} %
We start by defining in Section~\ref{sec:problem} an infinite-horizon optimal control problem for a discrete-time deterministic system, where the state is subjected to a conditional value-at-risk constraint at every time step. The optimal control problem encodes the task of taking the state of the system from a starting point to a target equilibrium. We present our main contribution in Section~\ref{sec:method} in the form of the distributionally robust iterative MPC scheme that progressively approximates the solution of the infinite-horizon problem.  In each iteration of our procedure, we generate a trajectory using an MPC scheme that takes the system from the start to the target state. The risk constraint at each iteration are imposed for all distributions contained in a general class of ambiguity sets defined using data collected in previous iterations. The terminal constraint of the MPC scheme enforces the state to lie in a subset of the states that were visited in previous iterations. We termed these states as sampled-safe-set. Once a trajectory is generated, the uncertainty samples collected in the iteration are added to the dataset, the ambiguity set is updated, and the sampled-safe-set is modified to only retain states that satisfy the risk constraint for the newly formed ambiguity set.

We establish several properties of our proposed method in Section~\ref{sec:properties} and~\ref{sec:one-step}. \emph{First,} under the assumption that a robustly feasible trajectory is available at the first iteration, we show that each iteration is recursively feasible and safe, where safety means satisfying the risk constraint with high probability. Further, we prove that each trajectory asymptotically converges to the target state. \emph{Second}, we identify conditions under which the set of sampled safe states grow at the end of an iteration and the trajectory cost decrease. A key requirement for this to happen is continuity of optimal finite-horizon cost. Therefore, the \emph{third} set of results analyze system properties that ensure such continuity. Since our method requires that each trajectory is terminates in finite time, in our \emph{fourth} set of results, we provide conditions on system dynamics, constraints, and cost function that guarantee finite-time convergence.

In addition to properties of our method, in Section~\ref{sec:comp-trac}, we provide tractable reformulations of the risk constraint for two distance-based ambiguity sets, those using total variation and Wasserstein metric. In the end, in Section~\ref{sec:sims}, we apply our algorithm to find risk-averse paths for two mobile robots in the presence of a randomly-moving obstacle.

A preliminary version of this work was published as~\cite{AZ-AC:22-ecc}. Compared to it, this work contains in addition the following results: (i) continuity of the optimal value of the optimization problem solved at each time-step (see Proposition~\ref{pr:across-iter} and~\ref{pr:cost-continuity}); (ii) the finite-time reachability of the target point (Section~\ref{sec:one-step}); and (iii) reformulations of the DR risk constraints (Section~\ref{sec:comp-trac}). Moreover, this version contains detailed proofs and presents a more elaborate simulation example.

\section{Preliminaries}\label{sec:prelims}
Here we collect notation and mathematical background.  
\subsubsection{Notation}\label{subsec:notation}
Let $\real$, $\realnonnegative$, and $\naturalnumbers$ denote the set of real, non-negative real, and natural numbers, respectively. The set of natural numbers excluding zero is denoted as $\naturalnumbers_{\ge 1}$. Let $\norm{\cdot}$ and $\norm{\cdot}_1$ denote the Euclidean $2$- and $1$-norm, respectively. For $N \in \naturalnumbers$, we denote $[N] := \{0,1,\dots,N\}$. Given $x \in \real$, we let $[x]_+ = \max(x,0)$. Given two sets $X$ and $Y$, a set-valued map $\setmap{f}{X}{Y}$ associates to each point in $X$ a subset of $Y$. The $n$-fold Cartesian product of a set $\SSs$ is denoted as $\SSs^n$. The number of elements in a set $\SSs$ is given by $\abs{\SSs}$. The $n$-dimensional unit simplex is denoted as $\Delta_{n}$.  The open ball of radius $\delta > 0$ centered at $x \in \real^n$ is denoted as $B_\delta(x) = \setdef{y \in \real^n}{\norm{x-y} < \delta}$. A function $\map{\alpha}{\realnonnegative}{\realnonnegative}$ is class
	$\KK$ if it is continuous, strictly increasing, and $\alpha(0) = 0$. The diameter of a closed set $\XX \subset \real^n$ is given as $\diam(\XX) := \max_{x_1, x_2 \in \XX} \norm{x_1 - x_2}$.

\subsubsection{Conditional Value-at-Risk}\label{subsec:cvar}
We review notions on conditional value-at-risk (CVaR) from~\cite{AS-DD-AR:14}. Given a real-valued random variable $Z$ with probability distribution $\Pb$ and $\beta \in (0,1)$, the \myemphc{value-at-risk} of $Z$ at level $\beta$, denoted $\VaR_\beta^{\Pb}[Z]$, is the left-side $(1-\beta)$-quantile of $Z$. Formally, $\VaR_\beta^\Pb [Z]  	 = \inf \setdef{\zeta}{\Pb(Z \le \zeta) \ge 1-\beta}$. 
The \myemphc{conditional value-at-risk (CVaR)} of $Z$ at level $\beta$, denoted $\CVaR_\beta^\Pb [Z]$, is given as 
\begin{align}\label{eq:cvar-def-alt}
	\CVaR_\beta^\Pb [Z] = \inf_{t \in \real} \Bigl\{ t + \beta^{-1} \Eb^\Pb[Z - t]_+ \Bigr\},
\end{align}
where $\Eb^\Pb[\,\cdot\,]$ denotes expectation under $\Pb$. Under continuity of the cumulative distribution function of $Z$ at $\VaR_\beta^\Pb[Z]$, we have 
	$\CVaR_\beta^\Pb [Z] := \Eb^\Pb[Z \ge \VaR_\beta^\Pb [Z]]$.
The parameter $\beta$ characterizes risk-averseness. When $\beta$ is close to unity, the decision-maker is risk-neutral, whereas, $\beta$ close to the origin implies high risk-averseness.

\section{Problem Statement} \label{sec:problem}

Consider the following discrete-time system:
\begin{equation}\label{sys}
    x_{t+1}  = f(x_t,u_t), 
\end{equation}
where $\map{f}{\real^{n_x} \times \real^{n_u}}{\real^{n_x}}$ defines the dynamics and $x_t\in\real^{n_x}$ and $u_t\in\real^{n_u}$ are the state and control input of the system at time $t$, respectively. The system state and control input are subject to the following deterministic constraints:
\begin{equation}
    x_t\in \XX , u_t\in \UU , \quad  \forall t\geq 0,
\end{equation}
where $\XX$ and $\UU$ are assumed to be \emph{compact} sets. The aim is to solve an infinite-horizon risk-constrained optimal control problem for system~\eqref{sys} that drives the system to a target equilibrium point $x_F \in \XX$. To that end, let $\map{r}{\XX \times \UU}{\realnonnegative}$ be a continuous function that represents the \emph{stage cost}  associated to the optimal control problem. We assume that %
\begin{align}\label{eq:st-cost}
\begin{cases}
r(x_F, 0) & = 0,\\
r(x, u) & > 0, \quad \forall (x,u) \in (\XX \times \UU) \setminus \{(x_F,0)\}.%
\end{cases}
\end{align}
Using this cost function, the \emph{risk-constrained infinite-horizon optimal control problem} is given as
\begin{subequations}\label{eq:IHOCP}
\begin{align}
	\min\quad  &\sum_{t=0}^{\infty}r(x_{t}, u_{t})\label{eq:IHOCP-obj} %
    \\
    \text{s.t.}  \quad & x_{t+1} = f(x_t,u_t), \quad \forall t\geq 0,\label{eq:IHOCP-a}
    \\
    \quad & x_t \in \XX , u_t\in \UU , \quad  \forall t\geq 0,\label{eq:IHOCP-c}
    \\
    \quad & x_0 = x_S,\label{eq:IHOCP-b}
    \\
    \quad & \CVaR_{\beta}^\mathbb{P}\left[g(x_t,w)\right] \leq \delta, \quad \forall t \geq 0, \label{eq:IHOCP-d}
\end{align}
\end{subequations}
where $x_S \in \XX$ is the initial state and constraint~\eqref{eq:IHOCP-d} represents the risk-averseness. Here, $\CVaR$ stands for the conditional value-at-risk (see Section~\ref{subsec:cvar} for details), $w$ is a random variable with distribution $\Pb$ supported on the compact set $\WW \subset \real^{n_w}$,  $\delta > 0$ is the risk tolerance parameter, 
$\beta > 0$ is the risk-averseness coefficient, and the continuous function $\map{g}{\XX \times \WW}{\real}$  is referred to as the constraint function. The constraint~\eqref{eq:IHOCP-d} ensures that the risk associated to the state at any time, as specified using the random function $g$, is bounded by a given parameter $\delta$. More generally, the constraint can be perceived as a safety specification for system~\eqref{sys} under uncertain environments. 

The infinite-horizon problem~\eqref{eq:IHOCP} is difficult to solve in general due to state, input, and risk constraints. Besides, in practice, the distribution $\Pb$ is usually unknown beforehand. To tackle these challenges, we propose a data-driven iterative MPC scheme outlined in the following section.

\section{Distributionally Robust Risk-constrained Iterative MPC}\label{sec:method}%
In this section, we provide an iterative strategy for solving the infinite-horizon optimal control problem~\eqref{eq:IHOCP} in an approximate manner. Here, each iteration refers to an execution of the control task, that is, taking the system state from $x_S$ to $x_F$ in a safe manner. Our iterative framework is inspired by~\cite{UR-FB:17-tac} and roughly proceeds in the following manner. At the start of any iteration $j$, we have access to a finite number of samples of the uncertainty, a set of safe states, and the cost it takes to go from each of these safe states to the target. In iteration $j$, we use this prior knowledge and define an MPC scheme that constructs a safe trajectory starting at $x_S$ and ending at $x_F$. The aim of this newly generated trajectory is to possibly reduce the cost or improve safety as compared to the previous iterations. At the end of the iteration, we update the dataset with samples gathered along the execution of the MPC scheme. Subsequently, we update the set of safe states. In the following, we make all the necessary ingredients of the iterative framework precise and later put them together in the form of  Algorithm~\ref{ag:DR_iteration}. 

\subsection{Components of the Iterative Framework}

\subsubsection{Trajectories}
Every iteration results in a trajectory. The system state and the control input at time $t$ of the $\jth$ iteration are denoted as $x_t^j$ and $u_t^j$, respectively, and the $\jth$ \emph{trajectory} is given by concatenated sets:
\begin{equation}\label{cl-traj}
	\begin{split}
		\xtraj^j& :=[x_0^j, x_1^j, \dots, x_t^j, \dots, x_{T_j}^j],\\
		\utraj^j& :=[u_0^j, u_1^j, \dots, u_t^j, \dots, u_{T_j - 1}^j].
	\end{split}
\end{equation}
We assume that all trajectories start from $x_S$, that is, $x_0^j = x_S$ for all $j \geq 1$. While our objective is to solve an infinite-horizon problem~\eqref{eq:IHOCP}, for practical considerations, we aim to find trajectories that reach the target $x_F$ in a finite number of steps. Thus, we assume that for each iteration $j$, the \emph{length} of the trajectory is finite, denoted by 
$T_j \in \naturalnumbers_{\ge1}$. Throughout the paper, whenever we mention trajectory of states, we implicitly mean that there exists a feasible control sequence that makes this trajectory of states possible. %

\subsubsection{Data and Ambiguity Sets}
At the start of iteration $j$,  a \emph{dataset} $\Data^{j-1} := \setr{\what_1, \dots, \what_{N_{j-1}}} \subset \WW$ of $N_{j-1}$ i.i.d. samples of the uncertainty $w$ drawn from $\Pb$ is available. Here, the index $j-1$ indicates the samples collected till iteration $j-1$. We assume that we collect one sample per time-step of each iteration and so the number of samples available for iteration $j+1$ are $N_j = N_{j-1} + T_j$. Our aim is to use the dataset $\Data^{j-1}$ to enforce the risk constraint~\eqref{eq:IHOCP-d} in an appropriate sense for the trajectory generated in the $\jth$ iteration. To this end, we adopt a distributionally robust approach. That is, we generate a set of distributions, termed \emph{ambiguity set}, that contains the underlying distribution $\Pb$ with high probability. We then enforce the risk constraint~\eqref{eq:IHOCP-d} for all distributions in the ambiguity set. To put the notation in place, assume that given a \emph{confidence parameter} $\zeta \in (0,1)$, we have access to a map   $\setmap{\Db}{\WW_\infty}{\PP(\WW)}$ such that given any set of $N$ i.i.d samples $\Data_N = \setr{\what_1, \dots ,\what_N}$ the set of distributions $\Db(\Data_N)$ contains $\Pb$ with confidence $\zeta$. In the definition of the map, the domain is $\WW_\infty = \cup_{i=1}^\infty \WW^i$ and $\PP(\WW)$ denotes the set of all distributions supported on $\WW$. We assume that $\Db$ always leads to a closed and nonempty ambiguity set.  We term $\Db$ as the \emph{ambiguity set generating map}. Given $\Db$, our strategy is to set the ambiguity set used for iteration $j$ as $\DD^{j-1}:=\Db(\Data^{j-1})$. The assumption on $\Db$ imply that $\DD^{j-1}$ is \emph{$(\zeta,\Pb^{\abs{\Data^{j-1}}})$-reliable}, that is, 
\begin{align}\label{eq:amb_def-n}
	\Pb^{\abs{\Data^{j-1}}} \left( \Pb\in\DD^{j-1} \right)\geq \zeta.
\end{align}
The above property implies that for the MPC scheme related to the $\jth$ iteration, if we impose the risk constraint~\eqref{eq:IHOCP-d} for all distributions in $\DD^{j-1}$, then the generated trajectory will satisfy the risk constraint with at least probability $\zeta$. Ideally, we must aim to find trajectories that satisfy~\eqref{eq:IHOCP-d}. However, when only limited data regarding the uncertainty is known, one can only enforce such a constraint in a probabilistic manner and the above definition aims to capture this feature. 

\subsubsection{Cost-to-go}
The cost-to-go from time $t$ for the trajectory $(\xtraj^j,\utraj^j)$ generated in iteration $j$, is denoted as:
\begin{align}\label{eq:to-go}
	\costgo{t}{\infty}^j &:=  \sum_{k=t}^\infty r(x_k^j, u_k^j).
\end{align}
Setting $t=0$ in~\eqref{eq:to-go} gives us the cost of the $\jth$ iteration as $\costgo{0}{\infty}^j$, that measures the performance of the controller in that iteration.
For every time-step $t \ge T_j$, we assume that the system remains at $x_F$ and the control input is zero. Thus, the infinite sum in~\eqref{eq:to-go} is well-defined due to~\eqref{eq:st-cost}.

\subsubsection{Sampled-safe-set}
The main advantage of the iterative scheme is that it allows data to be gathered and state space to be explored in an incremental manner. That is, we keep track of all samples from past iterations (discussed above) and we also maintain a set of safe states (along with the minimum cost that it takes to go to the target from them) that were visited in the previous iterations. These safe states are used in an iteration as terminal constraints in the MPC scheme (as proposed in~\cite{UR-FB:17-tac}).

In iteration $j$, the risk constraint~\eqref{eq:IHOCP-d} is imposed for all distributions in $\DD^{j-1}$ in the finite-horizon optimal control problem solved in the MPC scheme (see Section~\ref{sec:dr-finite-horizon}). Thus, due to~\eqref{eq:amb_def-n}, the trajectory $(\xtraj^j,\utraj^j)$ is \emph{$(\zeta,\Pb^{\abs{\Data^{j-1}}})$-safe}, that is 
\begin{align}\label{eq:traj-safety}
	\Pb^{\abs{\Data^{j-1}}} \left( \CVaR_{\beta}^\Pb \left[g(x^j_t,w)\right] \le \delta\right) \ge \zeta
\end{align}
for all $t \in [T_j]$. Note that $\xtraj^j$ is safe with respect to the dataset $\Data^{j-1}$. However, since the next iteration $j+1$ is built considering safety with respect to the dataset $\Data^j$, all previously generated trajectories need to be $(\zeta,\Pb^{\abs{\Data^{j}}})$-safe to be considered as the set of safe states in iteration $j+1$. In particular, the \emph{sampled-safe-set} obtained at the end of iteration $j$ and to be used in iteration $j+1$, denoted $\SSs^{j} \subseteq \until{j} \times \XX \times \realnonnegative$, is defined recursively as
\begin{align}\label{eq:SS-update-gen}
	\SSs^{j} = \mathbb{S}^j \Bigl(\SSs^{j-1} \cup \setr{(j,x_t^{j},\costgo{t}{\infty}^{j})}_{t=1}^{T_j} \Bigr).  
\end{align}
In the above expression, the set $\setr{(j,x_t^{j},\costgo{t}{\infty}^{j})}_{t=1}^{T_j}$ collects the set of states visited in iteration $j$, along with the associated cost-to-to. The counter $j$ is maintained in this set to identify the iteration to which a state with a particular cost-to-go is associated with. The set $\SSs^{j-1}$ is the sampled-safe-set used in iteration $j$. The map $\mathbb{S}^j$ only keeps the states that are safe with respect to the new data set $\Data^j$. This aspect of our method is different from~\cite{UR-FB:17-tac} where explored states are safe for all future iterations. The exact procedure that defines $\mathbb{S}^j$ is given in our algorithm.

For ease of exposition, we define maps $\Pit{\cdot}$, $\Pis{\cdot}$, and $\Pic{\cdot}$, such that, given a safe set $\SSs$,  $\Pit{\SSs}$, $\Pis{\SSs}$, and $\Pic{\SSs}$ return the set of all trajectory indices, states, and cost-to-go values that appear in $\SSs$, respectively.  The following assumption is required to initialize our iterative procedure with a nonempty sampled-safe-set. 

\begin{assumption}\longthmtitle{Initialization with robust trajectory}\label{assump:safeset}
	Before starting the first iteration, sampled-safe-set $\SSs^0$ contains the states of a finite-length robustly safe trajectory $\xtraj^0$ that starts from $x_S$ and reaches $x_F$. This means that the trajectory $\xtraj^0$ in $\SSs^0$ robustly satisfies all constraints of problem~\eqref{eq:IHOCP}, that is, %
		$x\in\XX,\,g(x, w) \leq \delta$ for all $w \in \WW$,
	and all $x \in \Pis{\SSs^0}$.\oprocend
\end{assumption}

\subsubsection{Minimum Cost-to-go}

Sampled-safe-set $\SSs^j$ keeps track of the cost-to-go associated with each state in the set. However, a state can appear in more than one trajectory.  For such cases, we need to maintain the minimum cost-to-go associated with a state. To this end, given the safe set $\SSs^j$ obtained at the end of iteration $j$, we define the associated minimum cost-to-go map as 

\begin{align}\label{eq:Qj}
	Q^j(x):=\begin{cases}
		\min\limits_{J\in F^j(x)}J, &\quad x\in\Pis{\SSs^j},
		\\
		+\infty, &\quad x\notin\Pis{\SSs^j},
	\end{cases}
\end{align}
where
\begin{align}\label{def:Fj}
	F^j(x) = \setdef{\costgo{t}{\infty}^i }{ \Pis{\left\{(i, x^i_t, \costgo{t}{\infty}^i)\right\}} = \{x\},&\nonumber
		\\ 
		(i, x^i_t, \costgo{t}{\infty}^i) \in \SSs^j}&.
\end{align}
In the above expression, the set $F^j(x)$ collects all cost-to-go values associated to the state $x \in \Pis{\SSs^j}$.  The function $Q^j$ then finds the minimum among these.

\subsubsection{DR Risk-constrained Finite-Horizon Problem}\label{sec:dr-finite-horizon}
Given the above described elements, we now present the finite-horizon optimal control problem solved at each time-step of each iteration. For generality, we write the problem for generic current state $x$, sampled-safe-set $\SSo$, and ambiguity set $\DDo$. Let $K \in \naturalnumbers_{\ge 1}$ be the length of the horizon and consider

\begin{equation}\label{eq:DR-RLMPC:main}
	\mathcal{J}_{(\SSo, \DDo)} (x) := \begin{cases}
		\min & \, \,  \sum_{k=0}^{K-1}r(x_{k}, u_{k}) +\overline{Q}(x_{K}) 
		\\
		\st & \, \, x_{k+1} = f(x_{k},u_{k}),   \forall k\in[K-1],
		\\
		& \, \, x_{k} \in\XX, u_{k}\in\UU,   \forall k\in[K-1],
		\\
		& \, \, x_{0}=x,
		\\
		& \, \, x_{K}\in\Pis{\overline{\SSs}}, 
		\\
		& \, \, \sup_{\mu\in\overline{\DD}} \left[\CVaR_{\beta}^{\mu}\left[g(x_{k},w)\right]\right]\leq \delta, 
		\\
		& \qquad \qquad \qquad \qquad \forall k\in[K-1],
	\end{cases}
\end{equation}
where $\map{\overline{Q}}{\XX}{\real}$ gives the minimum cost-to-go for all states in $\SSo$ and is calculated in a similar manner as in~\eqref{eq:Qj}. The decision variables in the above problem are $(x_0, x_1, \dots, x_K)$ and $(u_0, u_1, \dots, u_{K-1})$. The set $\SSo$ defines the terminal constraint $x_K \in \Pis{\SSo}$. Finally, note that the risk constraint is required to hold for all distributions in the ambiguity set $\overline{\DD}$. Thus, we refer to it as \emph{distributionally robust (DR) constraint}.  For iteration $j$ and time-step $t$, the MPC scheme solves the finite-horizon problem~\eqref{eq:DR-RLMPC:main} with $x = x_t^j$, $\SSo = \SSs^{j-1}$, $\DDo = \DD^{j-1}$, and $\overline{Q} = Q^{j-1}$.

\subsection{The Iterative Framework}
Here, we compile the elements described in the previous section and present our iterative procedure termed \emph{distributionally robust risk-constrained iterative MPC (DR-RC-Iterative-MPC)}. The informal description is given below. 
\begin{quote}
	\emph{[Informal description of Algorithm \ref{ag:DR_iteration}]:}
	Each iteration $j \ge 1$ starts with sampled-safe-set $\SSs^{j-1}$ and an ambiguity set $\DD^{j-1}$. The ambiguity set is constructed (see Line~\ref{ln:ambiguity}) using samples in dataset $\Data^{j-1}$ collected in previous iterations and the map $\Db$ that ensures~\eqref{eq:amb_def-n}. In the first step of the iteration (Line~\ref{ln:drmpc}), a trajectory $(\xtraj^j,\utraj^j)$ is generated by the $\mathtt{DR\_MPC}$ routine (described in Algorithm~\ref{ag:DR_MPC}) to which sampled-safe-set $\SSs^{j-1}$ and the ambiguity set $\DD^{j-1}$ are given as inputs. This trajectory is $(\zeta,\Pb^{\abs{\Data^{j-1}}})$-safe, that is, it satisfies~\eqref{eq:traj-safety}. The samples collected in iteration $j$ are appended to the dataset $\Data^{j-1}$ in Line~\ref{ln:drmpc} and the ambiguity set for the next iteration is constructed in Line~\ref{ln:ambiguity}. The trajectory $\xtraj^j$ along with its associated cost-to-go is appended to the sampled-safe-set in Line~\ref{ln:safe}. In Lines~\ref{ln:uns-traj} to~\ref{ln:ss-update}, sampled-safe-set $\SSs^{j-1}$ is modified to make sure that it only contains trajectories that are $(\zeta,\Pb^{\abs{\Data^{j}}})$-safe. These steps collectively represent the map $\mathbb{S}$ defined in an abstract manner in~\eqref{eq:SS-update-gen}. 	
	 The indices of trajectories present in $\SSs^{j-1}$ are maintained in the set $\II^{j-1}$. In Line~\ref{ln:uns-traj}, trajectories in $\II^{j-1} \cup \{j\}$ for which at least one state is not  $(\zeta,\Pb^{\abs{\Data^{j}}})$-safe are enumerated in the set $\Unsafe^{j}$. Accordingly, in Line~\ref{ln:M}, the set $\II^j$ is updated as trajectories in $\II^{j-1} \cup \{j\}$ that are not in $\Unsafe^j$. The states visited in these trajectories are stored in $\SSs^j$ in Line~\ref{ln:ss-update}. Finally, the minimum cost-to-go for all states in $\SSs^j$ is updated in Line~\ref{ln:Q}
\end{quote}
Note that in the above algorithm, the sampled-safe-set is updated in an iterative way. That is, we add the $\jth$ trajectory to $\SSs^{j-1}$ and then check safety with respect to the dataset $\Data^{j}$. In the process, we loose some trajectories in iterations $\{1,\dots, j-1\}$ that could have been $(\zeta,\Pb^{\abs{\Data^{j}}})$-safe. An alternative way would be to store all previous trajectories and check for their safety in each iteration. However, such a process would potentially slow down the algorithm.

\begin{algorithm}[htb]
	\SetAlgoLined
	\DontPrintSemicolon
	\SetKwInOut{Input}{Input}
	\SetKwInOut{Output}{Output}
	\SetKwInOut{init}{Initialize}
	\SetKwInOut{giv}{Data}
	\SetKwInOut{params}{Parameter}
	\Input{%
		\xvbox{2mm}{$\SSs^0$}\quad--$\;$Initial sampled-safe-set \\
		\xvbox{2mm}{$\Data^0$}\quad--$\;$Initial set of samples \\
		\xvbox{2mm}{$\II^0$}\quad--$\;$Index of trajectory in $\SSs^0$
		\\
	}
	\init{
		$j \gets 1$, $\DD^0 = \Db(\Data^0)$, $\Unsafe^{0}\gets\emptyset$}
	\While{$j > 0$ }
	{%
		Set $(\xtraj^{j}, \utraj^{j})\gets \mathtt{DR\_MPC}(\SSs^{j-1}, \DD^{j-1})$; $T^j\gets\mathtt{length}(\xtraj^{j})$; $\Data^{j} \gets \Data^{j-1} \cup \{\what_i\}_{i=1}^{T^j}$ \label{ln:drmpc} \;
		Set $\DD^j \gets\Db(\Data^j)$ \label{ln:ambiguity}\; 
		Set $\SSs^{j-1}\gets\SSs^{j-1} \cup \{(j, x_t^j, J^j_{(t:\infty)})\}_{t=1}^{T_j}$ \label{ln:safe}\;
			Set $\Unsafe^j \gets \{ i \in (\II^{j-1}\cup \{j\} ) \, |(i, x, J) \in \SSs^{j-1},$ $\sup\limits_{\mu \in \DD^{j}} \left[\CVaR_{\beta}^{\mu} \left[g(x,w)\right]\right] > \delta \}$ \label{ln:uns-traj} 
\;
		Set $\II^{j} \gets (\II^{j-1}\cup \{j\} ) \setminus \Unsafe^{j}$ \label{ln:M} \;
		Set $\SSs^{j} \gets \setdef{(i, x, J) \in \SSs^{j-1}}{i \in \II^{j}}$ \label{ln:ss-update}\;
		Compute $Q^j(x)$ for all $x\in\Pis{\SSs^j}$ using~\eqref{eq:Qj}\label{ln:Q}\;
		Set $j\gets j+1$
	}
	
	\caption{DR-RC-Iterative-MPC} %
	\label{ag:DR_iteration} 
\end{algorithm}

Algorithm~\ref{ag:DR_iteration} calls the $\drmpc$ routine in each iteration to generate the trajectory. This procedure is given in Algorithm~\ref{ag:DR_MPC} and informally described below.
\begin{quote}
	\emph{[Informal description of Algorithm~\ref{ag:DR_MPC}]:} 
	The procedure generates a trajectory from $x_S$ to $x_F$ given sampled-safe-set $\overline{\SSs}$ and an ambiguity set $\overline{\DD}$. The minimum cost-to-go function $\Qo$ is computed for $\SSo$ using~\eqref{eq:Qj}. At time-step $t$, problem~\eqref{eq:DR-RLMPC:main} is solved with $x = x_t$. We denote the optimal solution by 
	\begin{equation}\label{eq:optSol}
		\begin{split}
			x_{\mathrm{vec},t}^{*} &= [x_{t|t}^{*}, \dots , x_{t+K|t}^{*}], 
			\\
			u_{\mathrm{vec},t}^{*} &= [u_{t|t}^{*}, \dots , u_{t+K-1|t}^{*}],
		\end{split}
	\end{equation}
	where we use the notation that $x_{t+k|t}$ represents the prediction made at time $t$ regarding the state at time $t+k$. The control action at time $t$ is set as the first element of $u_{\mathrm{vec},t}^{*}$ (Line~\ref{ln:control}) and it is appended to the trajectory $\utraj$. The state is updated and added to $\xtraj$ in Line~\ref{ln:state}. The procedure moves to the next time step with the updated state as $x_{t+1}$. 
\end{quote}

\begin{algorithm}[htb]
	\SetAlgoLined
	\DontPrintSemicolon
	\SetKwInOut{Input}{Input}
    \SetKwInOut{Output}{Output}
	\SetKwInOut{init}{Initialize}
	\SetKwInOut{giv}{Data}
	\SetKwInOut{params}{Parameter}
	\SetKwProg{Fn}{Function}{:}{}
	\SetKwFunction{FMain}{$\mathtt{DR\_MPC}$}
    \Fn{\FMain{$\SSo, \DDo$}}{
    	\init{%
    		$t\gets0$; $x_0\gets x_S$; $\xtraj \gets[x_0]$, $\utraj \gets[\,\,]$
    	}
    Set $\overline{Q}$ as minimum cost-to-go in $\SSo$ (use~\eqref{eq:Qj})\;
    \While{$x_t\neq x_F$}{
        Solve~\eqref{eq:DR-RLMPC:main} with $x=x_t$ and obtain optimal solutions $x_{\mathrm{vec},t}^{*}$ and $u_{\mathrm{vec},t}^{*}$ (see~\eqref{eq:optSol})\; \label{step:finite-horizon}
        Set $u_t\gets u^{*}_{t|t}$; $\utraj \gets[u, u_t]$ \label{ln:control}\;
        Set $x_{t+1} \! \gets  \! f(x_t,u_t) $; $\xtraj \! \gets \! [x, x_{t+1}]$; $t \! \gets  \! t+1$ \label{ln:state}\;
        }
    \textbf{return} $(\xtraj, \utraj)$
    }
    \textbf{end}
    \caption{Distributionally robust MPC function}
    \label{ag:DR_MPC}
\end{algorithm}

The above explained MPC procedure might not terminate in finite time, thus possibly violating our assumption that all trajectories have finite length. To practically overcome this issue, we terminate the MPC scheme when the state reaches a neighborhood of the equilibrium $x_F$. Later in Section~\ref{sec:one-step}, we discuss conditions under which target can be reached in finite time.

\begin{remark}\longthmtitle{Tractability}
	Note that, if $\XX$ and $\UU$ are convex sets, $g(\cdot,w)$ is convex for every $w \in \WW$, and~\eqref{sys} is a linear system, then the risk constraint in the infinite-horizon problem~\eqref{eq:IHOCP-d} as well as the DR risk constraint in~\eqref{eq:DR-RLMPC:main} are convex. As a result of the latter fact, all points in the convex hull of $\Pis{\overline{\SSs}}$ satisfy the DR risk constraint. Hence, we can replace $\Pis{\overline{\SSs}}$ with its convex hull and define the minimum cost-to-go function using Barycentric functions (see~\cite{UR-FB:17-ifac}) in the problem~\eqref{eq:DR-RLMPC:main} without affecting the safety of the resulting trajectory. By doing so, all constraints in problem~\eqref{eq:DR-RLMPC:main} are convex which eases the computational burden of solving the problem. \oprocend
\end{remark}

\begin{remark}\longthmtitle{Ambiguity sets}\label{re:ambiguity}
	The definition of the ambiguity set in our algorithm is quite general, defined using an arbitrary map $\Db$. Popular choices of data-based ambiguity sets are the ones using distance metrics such as Wasserstein, KL-divergence, $\phi$-divergence or using moment information, see~\cite{HR-SM:19-arXiv} for a survey. The reliability guarantee~\eqref{eq:amb_def-n} for a particular choice of ambiguity set is ensured by concentration of measure results. Each class of ambiguity set comes with its own pros and cons and usually one needs to seek a balance between (a) guaranteed statistical performance, (b) generality of distributions that can be handled, and (c) computational effort for handling the DR constraint in the finite-horizon problem~\eqref{eq:DR-RLMPC:main}. In Section~\ref{sec:comp-trac}, we explore the tractability aspect of the DR constraint for two classes of distance-based ambiguity sets. \oprocend 
\end{remark}

\begin{remark}\longthmtitle{Safety vs cost-performance}\label{re:tradeoff}
	The reliability parameter $\zeta$ in our framework is tunable. Meaning, if one requires high level of safety when exploring the state space, then $\zeta$ can be selected close to unity. In that case, the ambiguity set needs to be large enough to ensure~\eqref{eq:amb_def-n} and so the DR risk constraint turns out to be conservative. Analogously, if cost improvement is the goal, then a low value of $\zeta$ will be sufficient. Note that safety can alternatively be tuned by changing the risk-averseness parameter $\beta$. A lower value of $\beta$ would ensure more safety.  \oprocend
\end{remark}

\begin{remark}\longthmtitle{Time-varying constraint function}
	We note that our procedure can handle time-varying constraint functions, that is, for time-step $t$, the risk-averse constraint~\eqref{eq:IHOCP-d} reads as $\CVaR_{\beta}^\mathbb{P}\left[g_t(x_t,w)\right] \leq \delta$. The only minor change in our algorithm will be in Line~\ref{step:finite-horizon} of the $\drmpc$ scheme. In particular, we need to replace $g$ in the distributionally robust constraint for the predicted state $x_k$ in~\eqref{eq:DR-RLMPC:main} with $g_{k+t}$ provided that the finite-horizon problem~\eqref{eq:DR-RLMPC:main} is solved at time-step $t$. The theoretical guarantees presented in the next section also hold for this general case provided that each $g_t$ has the same properties as $g$. \oprocend
\end{remark}

\section{Properties of DR-RC-Iterative-MPC}\label{sec:properties} %

We first establish recursive feasibility of our iterative procedure given in Algorithm~\ref{ag:DR_iteration}. We also state the safety guarantee with which each trajectory is generated.

\begin{proposition}\label{prop:t-t+1}\longthmtitle{Safety and recursive feasibility of DR-RC-Iterative-MPC}
	Let Assumption~\ref{assump:safeset} hold. Then, at each iteration $j\ge1$ and time-step $t \ge 0$, the finite-horizon problem~\eqref{eq:DR-RLMPC:main} with $x = x_t^j$, $\SSo = \SSs^{j-1}$, and $\DDo = \DD^{j-1}$ solved in the DR-RC-Iterative-MPC scheme is feasible. Further, the generated trajectory $(\xtraj^j,\utraj^j)$ is $(\zeta,\Pb^{\Data^{j-1}})$-safe.
\end{proposition}
\begin{proof}
	By Assumption~\ref{assump:safeset}, $\SSs^0$ contains a finite-length robustly safe trajectory from $x_S$ to $x_F$, denoted as, %
	\begin{align*}
		\xtraj^0 := [x_0^0, x_1^0, \dots, x_{T_0}^0], \, \text{ and } \, 
		\utraj^0 := [u_0^0, u_1^0, \dots, u_{T_0 - 1}^0],
	\end{align*} 
	where $T_0$ is the length of this trajectory. The update step for the sampled-safe-set (Line~\ref{ln:safe} and~\ref{ln:ss-update}) implies that 
		$\SSs^0 \subseteq \SSs^j$, for all $j\geq 1$. 
	Thus, at each iteration $j\geq1$ and $t = 0$, the first $K$ elements of $(\xtraj^0,\utraj^0)$ are valid feasible solutions to~\eqref{eq:DR-RLMPC:main}, where $x = x_S$ and $(\overline{\DD},\overline{\SSs}) = (\DD^{j-1},\SSs^{j-1})$.
	Our next step is to show that, for each iteration, feasibility at time $t$ implies feasibility at time $t+1$. The proof then follows by induction.
	
	Assume that the optimization problem~\eqref{eq:DR-RLMPC:main} is feasible at iteration $j$ and time $t$ for $x = x_t^j$. Denote the optimizer as
	\begin{equation}\label{eq:optSol-j}
			\begin{split}
				x_{\mathrm{vec},t}^{*,j} &= [x_{t|t}^{*,j}, \dots , x_{t+K|t}^{*,j}],
				\\
				u_{\mathrm{vec},t}^{*,j} &= [u_{t|t}^{*,j}, \dots , u_{t+K-1|t}^{*,j}].
			\end{split}
	\end{equation}
	By applying the first element of $u_{\mathrm{vec},t}^{*,j}$ to system~\eqref{sys}, the new state is determined as $x_{t+1}^j = f(x_{t}^j, u_{t|t}^{*,j})$. Moreover, we have $x_{t+1}^j = x_{t+1|t}^{*,j}$. 
	Due to constraints in~\eqref{eq:DR-RLMPC:main}, we have $x_{t+K|t}^{*,j} \in \SSs^{j-1}$. Recall that due to Line~\ref{ln:ss-update} of the algorithm, $\SSs^{j-1}$ contains trajectories that are $(\zeta,\Pb^{\Data^{j}})$-safe. Thus, there exists a trajectory starting at $x_{t+K|t}^{*,j}$ given as $\tilde{\xtraj} =  [x_{t+K|t}^{*,j}, \tilde{x}_{K+1}, \dots, \tilde{x}_{K+T}]$, $\tilde{\utraj} = [\tilde{u}_{K}, \tilde{u}_{K+1}, \dots, \tilde{u}_{K+T-1}]$, 
	such that $\tilde{x}_{K+T} = x_F$ and all points in $\tilde{\xtraj}$ are in $\Pis{\SSs^{j-1}}$. Using the above trajectory, the following one is a feasible solution to~\eqref{eq:DR-RLMPC:main} for time-step $t+1$, that is, when $x=x_{t+1}^j$: 
	\begin{align*}
		& [x_{t+1|t}^{*,j}, x_{t+2|t}^{*,j}, \dots, x_{t+K|t}^{*,j}, \tilde{x}_{K+1}], 
		\\
		& [u_{t+1|t}^{*,j}, u_{t+2|t}^{*,j}, \dots, u_{t+K-1|t}^{*,j}, \tilde{u}_{K}].
	\end{align*}
	This completes the proof of the first part. The safety of the trajectory $(\xtraj^j,\utraj^j)$ follows from the constraint in~\eqref{eq:DR-RLMPC:main} and the reliability assumption on the ambiguity set~\eqref{eq:amb_def-n}.
\end{proof}

Next, we show that each trajectory generated by Algorithm~\ref{ag:DR_MPC} converges asymptotically to $x_F$.
\begin{proposition}\label{pr:attr_xf}\longthmtitle{Convergence of $\drmpc$}
	Let Assumption~\ref{assump:safeset} hold. Then, for each iteration $j \ge 1$ of the DR-RC-Iterative-MPC procedure, the trajectory $(\xtraj^j,\utraj^j)$ generated by $\drmpc$ satisfies $x^j_t \to x_F$ as $t \to \infty$. 
\end{proposition}
\begin{proof}
	To show the result we will use Proposition~\ref{prop:xbar_conv}. Note that the trajectory belongs to the compact set $\XX$ and is generated by the discrete-time system $x^j_{t+1} = f(x_t^{j}, u^{*,j}_{\mathrm{vec},t})$, where $u^{*,j}_{\mathrm{vec},t}$ is the first component of the control input sequence obtained by solving the finite-horizon optimal control problem~\eqref{eq:DR-RLMPC:main} with $x = x_t^j$ (see~\eqref{eq:optSol-j} for more details). Note that the control input $u^{*,j}_{\mathrm{vec},t}$ is a function of the state $x_t^j$ only and so, the dynamics can be implicitly written in the form~\eqref{eq:gen-sys}. Next, consider the function $\map{\JJ_{(\SSs^{j-1},\DD^{j-1})}}{\XX}{\realnonnegative}$ obtained by replacing $(\SSo,\DDo)$ with $(\SSs^{j-1},\DD^{j-1})$ in the definition of $\JJ_{(\SSo,\DDo)}$ given in~\eqref{eq:DR-RLMPC:main}. For brevity, we use the shorthand $\JJ^{j-1}$. Our aim is to show that $\JJ^{j-1}$ acts as the Lyapunov candidate $V$. First note that $\JJ^{j-1}(x_F) = 0$ as one can apply zero input at the equilibrium and stay there, thus accumulating no cost. For any $x \not = x_F$, we have $\JJ^{j-1}(x) \ge \min_{u \in \UU} r(x,u) > 0$. Thus, in our case $x_F$ acts as the point $x^*$ in Proposition~\ref{prop:xbar_conv}. To conclude the proof, we show that $\JJ^{j-1}$ satisfies~\eqref{eq:V-lyap}. Pick any $x \in \XX$ and let $(\xtraj^*,\utraj^*) \in \XX^{K+1} \times \UU^K$ be the optimal trajectory obtained by solving the finite-horizon problem~\eqref{eq:DR-RLMPC:main} with the constraint that the first state in the trajectory is $x^*_0 = x$. We have
	\begin{align*}%
		\JJ^{j-1}(x_0^*) = r(x^*_0,u^*_0) + \sum_{k=1}^{K-1} r(x^*_k,u^*_k) + Q^{j-1}(x^*_{K}).
	\end{align*}
    By definition of $Q^{j-1}$, there exists a trajectory $j^* \in \until{j-1}$ and a time $t^*$ such that $x^*_{K} = x^{j^*}_{t^*}$ and 
    	$Q^{j-1}(x^*_{K}) = \sum_{t=t^*}^\infty r(x^{j^*}_t,u^{j^*}_t)$.
	Substituting this expression in the above relation for $\JJ^{j-1}$, gives 
	\begin{align}
		\JJ^{j-1}(x_0^*) & = r(x^*_0,u^*_0) + \sum_{k=1}^{K-1} r(x^*_k,u^*_k) + \sum_{t=t^*}^\infty r(x^{j^*}_t,u^{j^*}_t) \notag
		\\
		& = r(x^*_0,u^*_0) +  \sum_{k=1}^{K-1} r(x^*_k,u^*_k) + r(x^{j^*}_{t^*},u^{j^*}_{t^*}) \notag
		\\
		& \qquad \qquad +  \sum_{t=t^*+1}^\infty r(x^{j^*}_t,u^{j^*}_t) \label{eq:J-ineq-2}
	\end{align}
	Note that $x_{t^*+1}^{j^*}$ belongs to the set $\SSs^{j-1}$ and so 
	\begin{align}\label{eq:Q-bound}
		Q^{j-1}(x_{t^*+1}^{j^*}) \le \sum_{t=t^*+1}^\infty r(x^{j^*}_t,u^{j^*}_t),
	\end{align}
	and the finite-horizon trajectory $\xtraj^+ = [x^*_{1}, \dots, x^*_{K}, x^{j^*}_{t^*+1}]$, $\utraj^+ = [u^*_{1}, \dots, u^{j^*}_{t^*}]$ is a feasible solution for~\eqref{eq:DR-RLMPC:main} with ${x = x_1^*}$. 
	Using optimality, we have
	\begin{align*}
		& \JJ^{j-1}(x^*_1) \le \sum_{k=1}^{K-1} r(x^*_k,u^*_k) + r(x^*_{K},u^{j^*}_{t^*}) + Q^{j-1}(x^{j^*}_{t^* +1})
		\\
		& \quad \le \sum_{k=1}^{K-1} r(x^*_k,u^*_k) + r(x^{j^*}_{t^*},u^{j^*}_{t^*}) + \sum_{t=t^*+1}^\infty r(x^{j^*}_t,u^{j^*}_t),
	\end{align*}
	where we used~\eqref{eq:Q-bound} in the above inequality. 
	Using the above condition in~\eqref{eq:J-ineq-2} yields 
		$\JJ^{j-1}(x_0^*) \ge \JJ^{j-1}(x^*_1) + r(x^*_0,u^*_0)$.
	Since $r$ is a continuous function satisfying~\eqref{eq:st-cost}, this establishes the inequality~\eqref{eq:V-lyap} and so completes the proof. 
\end{proof}

While the above result establishes asymptotic convergence of the algorithm, it does not ensure that the trajectories are of finite length. The latter is a key assumption in our setup. Therefore, the following result identifies conditions under which trajectories reach any neighborhood of the target in finite steps. The proof is a direct consequence of~\cite[Theorem 2]{POMS-DQM-JBR:1999}.

\begin{lemma}\longthmtitle{Finite-time convergence of $\drmpc$ to a neighborhood of $x_F$}\label{le:finite-nbd}
	Assume that there exists a class $\mathcal{K}$ function $\rf$
	 such that $r(x, u)\ge\mathfrak{r}(\norm{x-x_F})$ for all $(x,u) \in \XX \times \UU$.
	 Then, for every $\epsilon > 0$ and for every iteration $j\ge1$ of the DR-RC-Iterative-MPC procedure generated by $\drmpc$, there exists a finite $\overline{t} \in\naturalnumbers_{\ge 1}$ such that $\norm{x_{\overline{t}}^j-x_F}\le\epsilon$.
\end{lemma}
In the subsequent section, we provide a set of stricter conditions as compared to the above result that ensure convergence to the target in finite time. For practical purposes, the above result helps to terminate each iteration in a reasonable number of time steps.

We now shift our attention to the 
performance of our algorithm across iterations. We state conditions under which the cost of trajectory decreases in subsequent iterations.

	\newcommand{\JJov}{\overline{\JJ}}
	\begin{proposition}\longthmtitle{Guarantee across iterations for DR-RC-Iterative-MPC}\label{pr:across-iter}
		For the DR-RC-Iterative-MPC procedure, if $\DD^{j} \subset \DD^{j-1}$ for some iteration $j \ge 1$, then we have 
		\begin{align}\label{eq:SS-over-iter}
			\SSs^j = \SSs^{j-1} \cup \setr{(j,x_t^{j},\costgo{t}{\infty}^{j})}_{t=1}^{T_j}.
		\end{align}
	As a consequence, $\SSs^{j-1}\subseteq\SSs^j$. In addition, if the function $\JJ_{(\SSs^{j},\DD^{j})}$ is continuous at $x_F$, then   $\costgo{0}{\infty}^{j+1} \le \costgo{0}{\infty}^{j}$. 
	\end{proposition}
\begin{proof}
	For convenience, we denote the function $\JJ_{(\SSs^{j},\DD^{j})}$ with $\JJov$. Let $(\xtraj^{j^*},\utraj^{j^*}),(\xtraj^{j+1^*},\utraj^{j+1^*}) \subset \XX \times \UU$ be the optimal trajectories obtained from Algorithm~\ref{ag:DR_MPC} at iterations $j$ and $j+1$, respectively. Based on~\eqref{eq:to-go} we have
	\begin{align}
			\costgo{0}{\infty}^{j} &:=  \sum_{k=0}^\infty r(x_k^{j^*}, u_k^{j^*}) \notag 
			\\
			&= \sum_{k=0}^{K-1} r(x_k^{j^*}, u_k^{j^*}) + \sum_{k=K}^\infty r(x_k^{j^*}, u_k^{j^*}). \label{eq:cost-ineq-j}
	\end{align}
	Due to the terminal constraint, we have $x_{K}^{j^*} \in \SSs^{j-1}$. Combining this with the assumption $\SSs^{j-1}\subseteq\SSs^{j}$, we have $x^{j^*}_{K}\in\Pis{\SSs^{j}}$. As a consequence, using the definition of $Q^j$, we have
	\begin{align*}
		Q^j(x_{K}^{j^*}) \le \sum_{k=K}^\infty r(x_k^{j^*}, u_k^{j^*}).
	\end{align*}
	Using the above in~\eqref{eq:cost-ineq-j}, we obtain %
\begin{align}\label{eq:jj_upper-pre}
	\costgo{0}{\infty}^{j} &\geq \sum_{k=0}^{K-1} r(x_k^{j^*}, u_k^{j^*}) +Q^{j}(x^{j^*}_{K})%
\end{align}

Note the point $([x_0^{j^*},x_1^{j^*},\dots, x_K^{j^*}],[u_0^{j^*},u_1^{j^*},\dots, u_{K-1}^{j^*}])$ form a feasible solution for the finite-horizon problem~\eqref{eq:DR-RLMPC:main} with $(\SSo,\DDo) = (\SSs^j,\DD^j)$. This is due to the assumption $\DD^j \subset \DD^{j-1}$ which ensures that these points satisfy the DR risk constraint for the $(j+1)$-th iteration. Using this fact and the definition of $\JJov$ in~\eqref{eq:jj_upper-pre},  we obtain
\begin{align}\label{eq:jj_upper}
	\costgo{0}{\infty}^{j} \geq \JJov(x_S), %
\end{align}
Thus, we have $\JJov(x_S)\leq \costgo{0}{\infty}^{j}$. Next, we find a lower bound for $\JJov(x_S)$.

From the proof of Proposition~\ref{pr:attr_xf} we have
	\begin{align*}
	\JJov(x_S) &\ge r(x^{j+1^*}_0,u^{j+1^*}_0) + \JJov(x^{j+1^*}_1)
	\\
	&\ge r(x^{j+1^*}_0,u^{j+1^*}_0) + r(x^{j+1^*}_1,u^{j+1^*}_1) + \JJov(x^{j+1^*}_2)
	\\
	&\ge \lim_{T\to\infty}\left[\sum_{t=0}^{T-1}r(x^{j+1^*}_t, u^{j+1^*}_t) + \JJov(x^{j+1^*}_T)\right].
\end{align*}
In Proposition~\ref{pr:attr_xf} we proved that for all $j\ge1$ $\lim_{T\to\infty}x^j_T = x_F$. Using this and the assumption that $\JJov$ is continuous at $x_F$, we obtain $\lim_{T\to\infty}\JJov(x^*_T) = 0$. Thus, we have 
\begin{align}\label{eq:jj_lower}
	\JJov(x_S)\ge\sum_{t=0}^{\infty}r(x^{j+1^*}_t, u^{j+1^*}_t) = \costgo{0}{\infty}^{j+1}.
\end{align}
According to~\eqref{eq:jj_upper}~and~\eqref{eq:jj_lower} we have $\costgo{0}{\infty}^{j+1}\le\JJov(x_S)\le\costgo{0}{\infty}^{j}$ and as a result, $\costgo{0}{\infty}^{j+1} \le \costgo{0}{\infty}^{j}$. This completes the proof.
\end{proof}
 The assumption that $\DD^j \subset \DD^{j-1}$ is a difficult one to impose, in general. In future, we would like to explore scenarios where this can be ensured at least with high probability if not almost surely. The above result also requires the optimal finite-horizon cost $\JJ_{(\SSs^{j},\DD^{j})}$ to be continuous at the target point $x_F$. While this property is in general true for standard MPC formulations, it is not straightforward in our scheme as the terminal set is non-convex. Hence, below we provide conditions under which continuity is guaranteed. %

\begin{proposition}\longthmtitle{Continuity of $\JJ^{j}$ at $x_F$}\label{pr:cost-continuity}
	Assume that for any neighborhoods $\XX_\gamma := \setdef{x \in \XX}{\norm{x-x_F}\le\gamma}$ and $\UU_\gamma := \setdef{u \in \UU}{\norm{u} \le \gamma}$ with $\gamma > 0$, there exists a neighborhood $\XX_\eta := \setdef{x \in \XX}{\norm{x - x_F} \le \eta}$ with $\eta > 0$ such that for any initial point $x \in \XX_\eta$, there exist  trajectories $[x_0, x_1, \dots, x_K] \subset \XX_\gamma$ and $[u_0, u_1, \dots, u_{K-1}] \subset \UU_\gamma$ that satisfy $x_0 = x$, $x_K = x_F$, $x_{i+1} = f(x_i, u_i)$, for all $i \in [K-1]$, and 
	\begin{align*}
		\sup_{\mu\in\DD^{j}} \left[\CVaR_{\beta}^{\mu}\left[g(x_i,w)\right]\right]\leq \delta,
	\end{align*}
	for all $i \in [K-1]$.
	Then, the function $\JJ^{j}$ is continuous at $x_F$, where $\JJ^{j}$ is the function $\JJ_{(\SSo,\DDo)}$ given in~\eqref{eq:DR-RLMPC:main} with $(\SSs^{j},\DD^{j})$ replacing $(\SSo,\DDo)$.
\end{proposition}

\begin{proof}
	From the hypothesis and by continuity of $r(\cdot, \cdot)$, we always can pick a $\gamma>0$ such that the optimal solution of the finite-horizon problem~\eqref{eq:DR-RLMPC:main} starting from any point in $\XX_\eta$ be a trajectory that ends at $x_F$.
	Keeping this fact in mind, consider the $(j+1)$-th iteration and the finite-horizon problem~\eqref{eq:DR-RLMPC:main} for some $x \in \XX_\eta$. Denoting the solution of the optimization problem as~\eqref{eq:optSol-j} with $j$ replaced with $j+1$, we have 
	\begin{align}\label{eq:JJ-no-Q}
		\JJ^{j}(x) = \sum_{i=t}^{t+K-1} r(x_{i|t}^{*,j+1}, u_{i|t}^{*,j+1}),
	\end{align} 
	 where $x_{0 | t}^{*,j+1} = x$.
		Hereafter, we use the expression for cost-to-go on $\XX_\eta$ given in~\eqref{eq:JJ-no-Q} to show that convergence of $x$ to $x_F$, results in the convergence of $\JJ^{j}(x)$ to $\JJ^{j}(x_F)=0$.
	
	For the sake of contradiction, assume that $\JJ^{j}$ is not continuous at $x_F$. This implies that there exists a sequence $\{x[\ell]\}_{\ell = 0}^\infty \subset \XX_\eta$ such that $\lim_{\ell \to \infty} x[\ell] = x_F$ and $\limsup_{\ell \to \infty} \JJ^{j} (x[\ell]) =: J_{\mathrm{sup}} > 0 = \JJ^{j}(x_F)$. For the sake of simplicity assume that the sequence $\{x[\ell]\}_{\ell =0}^\infty$ is such that $\lim_{\ell \to \infty} \JJ^{j} (x[\ell]) = J_{\mathrm{sup}}$. This is possible as otherwise one can always pick a subsequence of $\{x[\ell]\}_{\ell = 0}^\infty$ for which this is true. 
	By continuity of $r$ and the fact that $r(x_F,0) = 0$, we obtain that for every $m>0$ there exists $\gamma_m > 0$ such that if $(x,u) \in \XX_{\gamma_m} \times \UU_{\gamma_m}$, then $r(x,u) < \frac{m}{K}$. Set $m = J_{\mathrm{sup}}/2$. By hypothesis, there exists $\eta_m > 0$ such that if $x \in \XX_{\eta_m}$, then there exists a finite-length trajectory $([x_0,x_1, \dots, x_K],[u_0,u_1,\dots,u_{K-1}]) \subset \XX_{\gamma_m} \times \UU_{\gamma_m}$ with $x_0 = x$ and $x_K = x_F$ and using~\eqref{eq:JJ-no-Q}, we have
		\begin{align*}
			\JJ^j(x) = \sum_{k=0}^{K-1} r(x_k,u_k)  <  m = J_{\mathrm{sup}}/2.
		\end{align*} 
	That is, for all $x \in \XX_{\eta_m}$, we have $\JJ^j(x) < J_{\mathrm{sup}}/2$. %
	This contradicts the fact that $\lim_{\ell \to \infty} \JJ^{j} (x[\ell]) = J_{\mathrm{sup}}$ for the sequence $\{x[\ell]\}_{\ell =0}^\infty$ satisfying $x[\ell] \to x_F$.
\end{proof}

The hypothesis of the above result resembles those of stability conditions. That is, we require the system to be such that starting in a neighborhood of the target, the system can be steered to the target without drifting away from it and using small amount of control effort. These conditions might be difficult to verify and so, we provide an alternative continuity result in the following section. Furthermore, this section also provides conditions under which the system reaches the target in finite number of steps. These results contain a unifying assumption that the target can be reached locally in one step.

\section{Locally $1$-step reachable system}\label{sec:one-step}
Our method in Algorithm~\ref{ag:DR_iteration} requires trajectories to be of finite length. We discussed in Lemma~\ref{le:finite-nbd} how we can terminate each trajectory in finite-time by ensuring that it reaches a neighborhood of the target point in a finite number of steps. In this section, we go further and analyze the finite-time convergence to the target. We do this for the class of systems that can reach the target in one-step from an open neighborhood surrounding it.

We point to works \cite{POMS-DQM:1998}, \cite{POMS-DQM-JBR:1999}, \cite{ECK-JMM:2004}, and \cite{AA-AHG-AF-EK:2018} that investigate finite-time convergence of MPC schemes. Our result generalizes these works as we consider a nonlinear system (as opposed to works \cite{POMS-DQM:1998} and \cite{ECK-JMM:2004} that only consider linear system) and do not consider any region of attraction (as opposed to work \cite{POMS-DQM-JBR:1999} that proves finite-time convergence in the presence of a region of attraction). Also, we assume that any control input other than zero incurs a cost which generalizes the setup in~\cite{AA-AHG-AF-EK:2018}. %

We have the following standing assumption for this section which states that from every point in some neighborhood of the target, there exists a feasible control input that takes the system to the target. 
\begin{assumption}\longthmtitle{Locally $1$-step reachable}\label{as:one-step}
	The system~\eqref{sys} is locally $1$-step reachable. That is, there exists a $\gamma > 0$ such that for every $x \in \XX_\gamma := \setdef{x \in \XX}{\norm{x - x_F} \le \gamma}$ there exists $\uxone \in \UU$ such that $f(x,\uxone) = x_F$. \oprocend
\end{assumption}
As per the above assumption, in all the results of this section, $\XX_\gamma$ represents the set of points from which target can be reached in one step.

\begin{proposition}\longthmtitle{Finite-time convergence of $\drmpc$ to $x_F$}\label{prop:conv_x_f}
	Suppose Assumption~\ref{as:one-step} holds. Assume that the stage cost admits the separable form $r(x, u) := r_x(x) + r_u(u)$, where $r_x$ and $r_u$ are continuous functions satisfying  
	\begin{subequations}\label{eq:rx_ru_const}
	\begin{align}
			r_x(x_F) & = 0, \quad r_x(x)  > 0, \quad \forall x \in \XX  \setminus \{x_F\},
			\\
			r_u(0) & = 0, \quad	r_u(u)  > 0, \quad \forall u \in  \UU\setminus\{0\}.%
	\end{align}
\end{subequations}
	Assume that the following hold:
	\begin{enumerate}
		\item \label{hy:two} There exists $c > 0$ such that 
		\begin{align*}
			\XX_c := \setdef{x \in \XX}{\norm{x - x_F} \le c} \subset \QQ,
		\end{align*}
		where %
		$\QQ$ is the set of states in $\XX_\gamma$ that satisfy a particular cost inequality, specifically,
		\begin{align*}
			\QQ := &\setdef{x\in\XX_\gamma}{\exists \uxone\in\UU \text{ such that } f(x, \uxone) = x_F 
				\\
				& \quad \text{ and } r_x(f(x,u)) + r_u(u) \ge r_u(\uxone), \forall u\in\UU}.
		\end{align*} 
		\item There exists a class $\KK$ function $\rf$ such that $r_x(x)\ge \rf(\norm{x - x_F})$ for all $x \notin \QQ$.
	\end{enumerate}
	Then, for every iteration $j\ge1$ of the DR-RC-Iterative-MPC procedure generated by $\drmpc$, there exists a finite number $\overline{t}\in\naturalnumbers_{\ge 1}$ such that $x_{\overline{t}}^j = x_F$.
\end{proposition}
\begin{proof}
	We divide the proof into two steps. First, we show that starting from any point in $\XX$ and implementing the $\drmpc$ scheme, we reach the set $\QQ$ in a finite number of steps. Then, we establish that the optimal action from any point $\QQ$ is to reach the target $x_F$ in one step.

	For the first step, suppose that $x_S\notin\QQ$. 
	By hypothesis~\ref{hy:two}, we know that $\QQ$ contains an open neighborhood of $x_F$ and as a result there exists a strictly positive value $\overline{\rf} := \inf_{x\in\XX\setminus\QQ} \rf(\norm{x-x_F})$.
	Therefore, for all $x\notin\QQ$, $r_x(x)\ge\overline{\rf}$. From the proof of Proposition~\ref{pr:attr_xf}, for any trajectory $j$, we have the relation 
	\begin{align}\label{eq:next-iter-J}
		\JJ^{j-1}(f(x,u_x)) - \JJ^{j-1}(x)\le  -r(x, u_x),
	\end{align}
	where $u_x$ is the control input obtained by solving the finite-horizon problem for the state $x$. 
	Using the assumed separable from of $r$ and the fact that $r_u$ is nonnegative, we get
	\begin{align*}
		\JJ^{j-1}(f(x, u_x)) - \JJ^{j-1}(x) \le - r_x(x) \le - \overline{\rf}
	\end{align*}
	for all $j\ge1$ and $x \notin\QQ$. Let $\overline{k}$ be the smallest integer such that $\overline{k} \,\overline{\rf} > \JJ^{j-1}(x_S)$. Such an integer exists as $\overline{\rf}$ is positive and $\JJ^{j-1}(x_S)$ is finite. We assert that the state in iteration $j$ enters $\QQ$ in at most $\overline{k}$ number of steps. Indeed, otherwise, we have $r(x^j_k, u^j_k) \ge r_x(x^j_k) \ge\overline{\rf}$ for all $k\in[\overline{k}]$ and consequently applying~\eqref{eq:next-iter-J} iteratively, we have $\JJ^{j-1}(x^j_{\overline{k}}) \le \JJ^{j-1}(x_S) - \overline{k} \, \overline{\rf} < 0$ which is in contradiction with the fact that $\JJ^{j-1}$ is a nonnegative function. Thus, we conclude that if $x_S\notin\QQ$ then $x^j_{\overline{k}}\in\QQ$ for some finite integer $\overline{k}$, where the exact value of $\overline{k}$ is dependent on the iteration.

	Now, in the following, we proceed to the next step and show that we reach $x_F$ in one step from any $\QQ$. 
	At iteration $j$, consider any $x \in \QQ$ and let 
	\begin{align*}
		[x, \xt_1, \xt_2, \dots, \xt_{K}]
		\quad \text{ and } \quad [\ut_0, \ut_1, \ut_2, \dots, \ut_{K-1}]
	\end{align*}
	be the optimizer of the finite-horizon problem~\eqref{eq:DR-RLMPC:main} with $x_0 = x$. Then, we have
	\begin{align*}
		\JJ^{j-1}(x) = r(x,\ut_0) + \sum_{k=1}^{K-1} r(\xt_k,\ut_k) + 	Q^{j-1}(\xt_K).
	\end{align*}
	Using the fact that $r$, $r_u$, and $Q^{j-1}$ are nonnegative and $\tilde{x}_1 = f(x,\tilde{u}_0)$, we get
	\begin{align}
		\JJ^{j-1}(x) &\ge r(x, \ut_0) + r(f(x, \ut_0), \ut_1) \notag
		\\
		&\ge r_x(x) + r_u(\ut_0) + r_x(f(x, \ut_0)), \label{eq:ineq-uxone}
	\end{align}
	where we have used the separable form of $r$ and nonnegativity of $r_u$. By definition of $\QQ$, for the state $x \in\QQ$ for which we have written the above inequality, there exists an input $\uxone\in\UU$ such that $f(x, \uxone)=x_F$ and $r_u(\uxone)\le r_x(f(x, u))+r_u(u)$ for all $u\in\UU$.  Selecting $u = \tilde{u}_0$ in this inequality we get $r_u(\uxone) \le r_u(\tilde{u}_0) + r_x(f(x,\tilde{u}_0))$. Adding $r_x(x)$ to both sides of this inequality yields
	\begin{align*}
		r_x(x) + r_u(\uxone) & \le r_x(x) + r_u(\tilde{u}_0) + r_x(f(x,\tilde{u}_0)) 
		\\
		& \le \JJ^{j-1}(x)
	\end{align*}
	where the last inequality follows from~\eqref{eq:ineq-uxone}. Using the fact that $\JJ^{j-1}(x)$ is the optimal value of the finite-horizon problem, the above inequality implies that $\tilde{u}_0 = \uxone$. 
	This further implies that 
	for every $j\ge1$, when the state enters $\QQ$ the optimum solution is to reach $x_F$ in one step. As a result, in any iteration $j \ge 1$, $x_F$ is reachable in $\overline{k}+1$ steps from $x_S\notin\QQ$, concluding the proof.
\end{proof}

In the above result, the hypothesis on the problem data for ensuring finite-time convergence is rather implicit. That is, we require the set $\QQ$ to be nonempty and contain the target in the interior. This set stands for the set of states from which the optimal solution of the finite-horizon problem is to reach the target in one step. Below we identify sufficient conditions on the dynamics and the stage cost that guarantee the existence of such a set, thereby making it easier to check if the system converges in finite time.

\begin{proposition}\longthmtitle{Sufficient conditions for existence of $\QQ$ given in Proposition~\ref{prop:conv_x_f}}\label{lem:q_conds}
	Suppose Assumption~\ref{as:one-step} holds. Assume that for system~\eqref{sys} and stage cost~\eqref{eq:rx_ru_const} there exists $c>0$ and $\alpha_1, \alpha_2, \alpha_3, C_1, C_2, C_3 \ge 0$  such that we have %
	\begin{enumerate}
		\item $\norm{f(x, u_1)-f(x, u_2)}\ge C_1\norm{u_1-u_2}^{\alpha_1}$ for all $u_1, u_2\in\UU$ and for all $x\in \XX_c$, where $\XX_c = \setdef{x \in \XX}{\norm{x - x_F} \le c}$.\label{qq_assump1}
		\item $\abs{r_u(u_1)-r_u(u_2)}\le C_2\norm{u_1-u_2}^{\alpha_2}$ for all $u_1, u_2\in\UU$.\label{qq_assump2}
		\item $r_x(x)\ge C_3\norm{x-x_F}^{\alpha_3}$ for all $x\in \XX$.\label{qq_assump3}
	\end{enumerate}
	Assume that either of the following hold:
	\begin{enumerate}[label=(\alph*)]
		\item \label{qq_assump4_1} $\alpha_1 \alpha_3 = \alpha_2$ and $C_3C_1^{\alpha_3} \ge C_2$, or 
		\item \label{qq_assump4_2} $\alpha_1 \alpha_3 < \alpha_2$ and $\diam(\UU) \le (\frac{C_3C_1^{\alpha_3}}{C_2})^{\frac{1}{\alpha_2-\alpha_1\alpha_3}}$, where recall that $\diam(\UU)$ stands for the diameter of $\UU$. 
	\end{enumerate} 
	Then, $\XX_c \subseteq \QQ$. %
\end{proposition}
	\begin{proof}
		Pick any $x \in \XX_c$. Without loss of generality, there exists $\uxone \in\UU$ such that $f(x, \uxone)=x_F$. From hypothesis~\ref{qq_assump3} we have for any $u \in \UU$, 
		\begin{align*}
			r_x(f(x, u)) \ge C_3\norm{f(x,u)-x_F}^{\alpha_3}. 
		\end{align*}
		Further, by hypothesis~\ref{qq_assump2},
		\begin{align*}
			r_u(u) \ge r_u(\uxone) -C_2\norm{u-\uxone}^{\alpha_2}.
		\end{align*}
		Adding the above two inequalities, we get
		\begin{align*}
			r_x(f(x, u)) + r_u(u)\ge& C_3\norm{f(x, u)-f(x, \uxone)}^{\alpha_3}
			\\
			& \quad +r_u(\uxone)-C_2\norm{u-\uxone}^{\alpha_2}.
		\end{align*}
		Using hypothesis~\ref{qq_assump1} in the above inequality yields
		\begin{align*}
			r_x(f(x, u)) + r_u(u)\ge&
			C_3C_1^{\alpha_3}\norm{u-\uxone}^{\alpha_1\alpha_3}+r_u(\uxone)
			\\
			&-C_2\norm{u-\uxone}^{\alpha_2}.
		\end{align*}
		Now, by considering either condition~\ref{qq_assump4_1} or condition~\ref{qq_assump4_2} on the constants, we have
		\begin{align*}
			C_3C_1^{\alpha_3}\norm{u-\uxone}^{\alpha_1\alpha_3}-C_2\norm{u-\uxone}^{\alpha_2} \ge 0,
		\end{align*}
		for all $u$ and thus, $r_x(f(x, u)) + r_u(u)\ge r_u(\uxone)$ for all $u$. As a result, $x \in \QQ$ and so $\XX_c \subseteq \QQ$. 
	\end{proof}

\begin{remark}\longthmtitle{Sufficient conditions finite-time convergence}
	Hypothesis (i), (ii), and (iii) of Lemma~\ref{lem:q_conds} and conditions on the coefficients involved there reflect different aspects of problem data that together guarantee convergence of the trajectory in finite time. Condition (iii) enforces the state-cost to have enough growth rate so that the optimal input drives the system to the target in one step. On the other hand, condition (ii) affirms that the growth rate of the input-cost is bounded so that the optimal input is allowed to take high enough values, again to get to the target in one step. Finally, for these two cost related growth conditions to work together, we require the system dynamics to be sensitive with respect to the input in the neighborhood of the target, which is what condition (i) stipulates. 
	\oprocend
\end{remark}
In the final result of this section, we show how the conditions imposed in Proposition~\ref{prop:conv_x_f} for finite-time convergence also simplify the conditions for continuity of the function $\JJ^{j}$.

\begin{proposition}\longthmtitle{Continuity of $\JJ^{j}$ at $x_F$}\label{pr:continuity}
	Let the stage cost $r$ be of the form~\eqref{eq:rx_ru_const} and hypothesis (i) and (ii) of Proposition~\ref{prop:conv_x_f} hold. Assume that $f$ is continuous on $\QQ\times\UU$. Then, for any $j \ge 1$, the function $\JJ^{j}$ is continuous at $x_F$, where $\JJ^{j}$ is the function $\JJ_{(\SSo,\DDo)}$ given in~\eqref{eq:DR-RLMPC:main} with $(\SSs^{j},\DD^{j})$ replacing $(\SSo,\DDo)$.
\end{proposition}
\begin{proof}
	Pick any $j$. From the proof of Proposition~\ref{prop:conv_x_f}, under the assumed conditions on the stage cost, we know that the optimal solution of the finite-horizon problem~\eqref{eq:DR-RLMPC:main} starting from any point in $\QQ$ is to reach the target $x_F$ in one step. Thus, $\JJ^{j}(x) = r(x,\uxone)$ for all $x \in \QQ$, where $\uxone$ is a control input that takes the system to $x_F$ from $x$ and satisfies the condition in hypothesis (i) of Proposition~\ref{prop:conv_x_f}.
	
	Consider any sequence $\{x[\ell]\}_{\ell =0}^\infty \subset \QQ$ such that $\lim_{\ell \to \infty} x[\ell] = x_F$. For each $\ell$, let $u^{\mathrm{one}}[\ell] \in \UU$ be a control input such that $f(x[\ell],u^{\mathrm{one}}[\ell]) = x_F$ and $\JJ^{j}(x[\ell]) = r(x[\ell],u^{\mathrm{one}}[\ell])$. Since the sequence $\{u^{\mathrm{one}}[\ell]\}$ belongs to a compact set, there exists a convergent subsequence, that is, $\lim_{k\to\infty} u^{\mathrm{one}}[\ell_k] = \bar{u}$. We first show that $\bar{u} = 0$. To this end, from continuity of $f$ on $\QQ\times \UU$ we have 
	\begin{align*}
		x_F = \lim_{k \to \infty} f(x[\ell_k],u^{\mathrm{one}}[\ell_k]) & = f(\lim_{k \to \infty} (x[\ell_k],u^{\mathrm{one}}[\ell_k] )) 
		\\
		&= f(x_F,\bar{u}),
	\end{align*}
	which clearly implies that $\bar{u} = 0$. Thus, if a subsequence of $\{u^{\mathrm{one}}[\ell]\}$ converges, then the accumulation point is $0$. For the sake of contradiction, assume that $\JJ^j$ is not continuous at $x_F$ and so
	\begin{align*}
		\limsup_{\ell \to \infty} \JJ^{j}(x[\ell]) := J_{\mathrm{sup}} > 0 = \JJ^{j}(x_F).
	\end{align*}
	Consider a subsequence $\{x[\ell_k]\}_{k=0}^\infty$ such that $\lim_{k \to \infty} \JJ^{j}(x[\ell_k]) = J_{\mathrm{sup}}$ and $\lim_{k \to \infty} u^{\mathrm{one}}[\ell_k] =0$. Then, by continuity of $r$, we have 
	\begin{align*}
		0 < J_{\mathrm{sup}} = \lim_{k \to \infty} \JJ^{j}(x[\ell_k]) & = \lim_{k \to \infty} r(x[\ell_k],u^{\mathrm{one}}[\ell_k]) 
		\\
		& = r \left( \lim_{k \to \infty} (x[\ell_k],u^{\mathrm{one}}[\ell_k]   ) \right) 
		\\
		& = r(x_F,0 ) = 0.
	\end{align*}
	This is a contradiction and so, $\JJ^{j}$ is continuous at $x_F$.
\end{proof}

\begin{remark}\longthmtitle{Comparing  Proposition~\ref{pr:cost-continuity} and~\ref{pr:continuity}}\label{re:cont-comparision}
	We have gathered two sets of conditions, in Proposition~\ref{pr:cost-continuity} and~\ref{pr:continuity}, under which the function $\JJ_{(\SSs^{j},\DD^{j})}$ is continuous at the target state $x_F$. In comparison, the former result imposes implicit conditions on the system dynamics that are akin to requirements of stability. That is, starting from a point in a neighborhood in $x_F$, we can reach the target in finite number of steps, while remaining ``close'' to  $x_F$ and putting in ``minimal'' control effort. Regarding the stage-cost $r$, this result does not ask for anything more than continuity.  On the other hand, Proposition~\ref{pr:continuity} requires the system to be locally 1-step reachable and the cost to satisfy assumptions in Proposition~\ref{prop:conv_x_f}. However, the conditions specified in Proposition~\ref{pr:continuity} are easy to verify using system and cost specifications, for instance, by making use of Proposition~\ref{lem:q_conds}. 
	\oprocend
\end{remark}

\section{Computational Tractability of DR Constraint}\label{sec:comp-trac}
A closer look at the optimization problem~\eqref{eq:DR-RLMPC:main} reveals that the DR risk constraint is non-trivial to handle due to the maximization over the space of  probability distributions. In this section, considering two different types of data-driven ambiguity sets, we present computationally tractable reformulations of the DR constraint. %
All ambiguity sets are characterized by distributions that are close to the empirical distribution in some metric. We term them as distance-based ambiguity sets. Further, we assume that the support $\WW$ of all distributions contains a finite number $L$ of points, that is, $L = \abs{\WW}$. Given $N$ samples $\setr{\what_1, \dots, \what_N}$ of the uncertainty, the empirical distribution is given by the vector $\widehat{\mathbb{P}}^{N} := (p_\ell^{N})_{\ell=1}^{L}$, where $p_\ell^N = (\text{frequency of }w_\ell \in \WW \text{ in the dataset})/N$. Using this definition, a distance-based ambiguity set is of the form 
\begin{align}\label{eq:sim-ambiguity}
	\DD = \setdef{\mu \in \Delta_{L}}{\delta(\mu, \widehat{\mathbb{P}}^{N})\le \theta},
\end{align}
where $\delta(\cdot)$ is the distance function and $\theta \ge 0$ is the radius. Given samples of the uncertainty and a confidence level $\zeta \in (0,1)$, one can tune $\theta$ to obtain a reliability bound as~\eqref{eq:amb_def-n}. The distance functions considered in this section are \emph{total variation} and \emph{Wasserstein metric}. 
 
\subsection{Total variation}\label{sec:TV-reform} For distributions $P, Q \in \Delta_{L}$ supported on a finite set $\WW$, the total variation (TV) distance between them is defined as $\delta_{\ell_1}(P, Q) = \frac{1}{2}\norm{P-Q}_1$. Thus, TV-based ambiguity set is defined as $\DD_{\ell_1} = \setdef{\mu \in \Delta_{L}}{\delta_{\ell_1}(\mu, \widehat{\mathbb{P}}^{N})\le \theta}$. 
Below we provide a tractable reformulation of the DR constraint for this ambiguity set.
\begin{proposition}\longthmtitle{Reformulation for TV distance}\label{prop:TV-dual}
	The following holds:
\begin{align} \label{eq:reform-tv}
	&\sup_{\mu\in \DD_{\ell_1}} \CVaR_{\beta}^{\mu}\left[g(x,w)\right] = 
	\\
	&\begin{cases} \nonumber \inf & \quad 2 \lambda \theta + \eta + \nu + \sum_{\ell = 1}^L (\gamma_{1_\ell}-\gamma_{2_\ell})p_\ell^{N} 
		\\
		\st  &\quad \beta( 
		\gamma_{1_\ell}-\gamma_{2_\ell}+\nu) \geq [g(x,w_\ell)-\eta]_+, \quad\forall\ell\in[L], 
		\\
		&\quad 
		\gamma_{1_\ell}+\gamma_{2_\ell} = \lambda, \quad\forall\ell\in [L], 
		\\
		& \quad 
		\eta, \nu\in\real, \quad \lambda, \gamma_{1_\ell}, \gamma_{2_\ell} \in \realnonnegative ,\quad\forall\ell\in [L].
	\end{cases} 
\end{align}
\end{proposition}
\begin{proof}
By the definition of $\CVaR$, we have
\begin{align}
		\nonumber&\sup_{\mu\in \DD_{\ell_1}} \CVaR_{\beta}^{\mu}\left[g(x,w)\right]
		\\
		\nonumber&  = \sup_{\mu\in \DD_{\ell_1}}\inf_{\eta\in\real}\left\{\eta+\frac{1}{\beta}\mathbb{E}^\mu[g(x,w)-\eta]_+\right\}
		\\
		\nonumber& \overset{(a)}{=} \inf_{\eta\in\real}\sup_{\mu\in \DD_{\ell_1}}\left\{\eta+\frac{1}{\beta}\mathbb{E}^\mu[g(x,w)-\eta]_+\right\}
		\\
		\nonumber& \overset{(b)}{=} \inf_{\eta\in\real}\sup_{\mu\in \Delta_{L}} \Bigl\{ \eta+\frac{1}{\beta}\mathbb{E}^\mu[g(x,w)-\eta]_+ \, \Big| \, \frac{1}{2}\norm{\mu-\widehat{\mathbb{P}}^{N}}_1\le \theta \Bigr\},
		\\
		\nonumber& \overset{(c)}{=} \inf_{\substack{\eta\in\real \\ \lambda\ge0}} \sup_{\mu\in \Delta_{L}} \left\{ \eta+\frac{1}{\beta}\mathbb{E}^\mu[g(x,w)-\eta]_+ \right.
		\\
		& \qquad \qquad \qquad \qquad \qquad   \left. +\lambda(2 \theta-\norm{\mu-\widehat{\mathbb{P}}^{N}}_1)\right\},\label{eq:cvar_def}
\end{align}
where (a) follows from~\cite[Theorem 3]{AC-ARH:2020}, (b) uses the definition of the ambiguity set, and (c) follows from the equivalence of the inner max optimization problem in (b) with its dual form~\cite[Chapter 5.2]{SB-LV:04}. 
Hereafter, the focus is on reformulating the inner supremum in the above obtained inf-sup problem. 
Using the definition of expected value and $1$-norm, we get 
\begin{align}\label{eq:innersup}
	&\sup_{\mu\in \Delta_{L}} \left\{ \frac{1}{\beta}\mathbb{E}^\mu[g(x,w)-\eta]_+ -\lambda\norm{\mu-\widehat{\mathbb{P}}^{N}}_1 \right\} \nonumber
	\\
	&\quad = \sup_{\mu\in \Delta_{L}} \left\{\sum_{\ell=1}^L\mu_\ell\frac{[g(x,w_\ell)-\eta]_+}{\beta} -\lambda\abs{\mu_\ell-p^{N}_\ell} \right\}
	\\
	& \quad = \begin{cases} \sup_{\mu, z} & \quad \sum_{\ell=1}^La_\ell\mu_\ell-\lambda z_\ell 
		\\
	\st & \quad z_\ell\geq \mu_\ell-p^{N}_\ell, \forall \ell \in [L],
\\
& \quad z_\ell \geq p^{N}_\ell - \mu_\ell, \forall \ell \in [L],
\\
& \quad \sum_{\ell=1}^L \mu_\ell = 1,
\\
& \quad \mu_\ell \ge 0, \forall \ell \in [L],\end{cases}
\end{align}
where in the last inequality $a_\ell := \frac{[g(x,w_\ell)-\eta]_+}{\beta}$ for all $\ell \in \{1,\dots,L\}$.
The Lagrangian of the above optimization problem is written as:
\begin{align*}
	&L(\mu, z, \gamma_1, \gamma_2, \nu,\tau) \! = \! \sum_{\ell=1}^L \! \left(a_\ell\mu_\ell-\lambda z_\ell\right) \! + \! \sum_{\ell=1}^L \! \gamma_{1\ell} \left(z_\ell - \mu_\ell + p^N_\ell \right)
	\\
	& \quad + \sum_{\ell = 1}^L \gamma_{2\ell} \left(z_\ell - p^N_\ell - z_\ell \right) + \nu \left( \sum_{\ell=1}^L \mu_\ell - 1 \right) + \sum_{\ell = 1}^L \tau_\ell \mu_\ell,
\end{align*}
where we used the shortcut $\gamma_1 = (\gamma_{1\ell})_{\ell \in [L]}$ and $\gamma_2 = (\gamma_{2\ell})_{\ell \in [L]}$. Using the above definition, the dual of~\eqref{eq:innersup} is
\begin{align*}%
	\inf  & \quad \sum_{\ell = 1}^L  \left( \gamma_{1\ell} - \gamma_{2\ell} \right)p^N_\ell -\nu 
	\\
	\st & \quad a_\ell - \gamma_{1\ell} + \gamma_{2 \ell} + \nu + \tau_\ell = 0, \quad  \forall \ell \in [L],
	\\
	& \quad \gamma_{1\ell} + \gamma_{2\ell} = \lambda, \quad \forall \ell \in [L],
	\\
	& \quad \nu \in \real, \quad \gamma_{1\ell}, \gamma_{2\ell}, \tau_\ell \ge 0, \quad \forall \ell \in [L]. 
\end{align*}
Using the above dual form of~\eqref{eq:innersup} in~\eqref{eq:cvar_def} and subsequently, using the definition of $a_\ell$ and eliminating $\tau_\ell$ by making the constraint that it involves as inequality, leads to the desired reformulation. This completes the proof. 
\end{proof}

The above result can be used to bring the problem~\eqref{eq:DR-RLMPC:main} to a tractable form. In particular, the DR constraint in there can be replaced with the a set of constraints that are obtained by upper bounding the objective function of the right-hand side of~\eqref{eq:reform-tv} and appending the constraints of the said problem. That is, 
\begin{align*}
	\begin{cases}&2 \lambda \theta + \eta + \nu + \sum_{\ell = 1}^L (\gamma_{1_\ell}-\gamma_{2_\ell})p_\ell^{N} \le \delta,	
	\\
	& \beta( 
	\gamma_{1_\ell}-\gamma_{2_\ell}+\nu) \geq [g(x,w_\ell)-\eta]_+, \quad\forall\ell\in[L], 
	\\
	& 
	\gamma_{1_\ell}+\gamma_{2_\ell} = \lambda, \quad\forall\ell\in [L], 
	\\
	&  
	\eta, \nu\in\real, \quad \lambda, \gamma_{1_\ell}, \gamma_{2_\ell} \in \realnonnegative ,\quad\forall\ell\in [L].
	\end{cases}
\end{align*}

\subsection{Wasserstein}
The Wasserstein distance (of order 1) between two discrete distributions $P, Q \in \Delta_{L}$ is defined as $\delta_{\Wb}(P, Q) =\inf_{\kappa\in\HH(P, Q)}\sum_{w_1 \in \WW, w_2 \in \WW }\norm{w_1-w_2}\kappa(w_1, w_2)$, where $\HH(P, Q)$ is the set of all distributions with marginals $P$ and $Q$, and $\kappa(w_1, w_2)$ is the probability of point $(w_1,w_2) \in \WW \times \WW$ under the distribution $\kappa$. Note that here $\kappa$ is represented by a square matrix of size $L$ where the summation of rows and columns lead to distributions $Q$ and $P$, respectively. The Wasserstein-based ambiguity set is defined as $\DD_{\Wb} = \setdef{\mu \in \Delta_{L}}{\delta_{\Wb}(\mu, \widehat{\mathbb{P}}^{N})\le \theta}$. The next result gives a reformulation of DR constraint under $\DD_{\Wb}$. The proof follows using similar arguments as in~\cite{AH-IY:20-ICRA} and~\cite{BP-AH-SB-DC-AC:20}.

\begin{proposition}\longthmtitle{Reformulation for Wasserstein distance}\label{prop:Wass-dual}
	The following holds:
	\begin{align*}
		&\sup_{\mu\in \DD_{\Wb}} \CVaR_{\beta}^{\mu}\left[g(x,w)\right] = 
		\\
		&\begin{cases} \nonumber\inf  & \quad \lambda\theta + \eta +   \sum_{\ell = 1}^L s_\ell
			\\
			\st   
			& \quad p_\ell^N \left( [g(x,w_i)-\eta]_+ - \lambda \norm{w_i - w_\ell}\right) \le \beta s_\ell, 
			\\
			& \qquad \qquad \qquad \qquad  \qquad \qquad \forall i \in [L], \ell \in [L],
			\\
			& \quad \lambda \in \realnonnegative, \quad \eta, s_\ell \in \real, \forall \ell \in [L]. %
		\end{cases}
	\end{align*}
\end{proposition}

\section{Simulation}\label{sec:sims}
In this section, Algorithm~\ref{ag:DR_iteration} is applied in form of a centralized control strategy to solve a motion planning problem for two mobile robots in a 2D environment that contains a randomly moving obstacle. With the progress of iterations, more data is collected and the safe set is gradually expanded. This improves the cost-performance of the system while maintaining the required safety level.
\subsubsection{Setup}
Consider the following model for two circular mobile robots navigating in a 2D environment as a single system:
\begin{align*}
    x_{t+1} & = \begin{bmatrix}
            1 & 0 & 1 & 0 & 0 & 0 & 0 & 0 \\
            0 & 1 & 0 & 1 & 0 & 0 & 0 & 0 \\
            0 & 0 & 1 & 0 & 0 & 0 & 0 & 0 \\
            0 & 0 & 0 & 1 & 0 & 0 & 0 & 0 \\
            0 & 0 & 0 & 0 & 1 & 0 & 1 & 0 \\
            0 & 0 & 0 & 0 & 0 & 1 & 0 & 0 \\
            0 & 0 & 0 & 0 & 0 & 0 & 1 & 0 \\
            0 & 0 & 0 & 0 & 0 & 0 & 0 & 1  
    \end{bmatrix}x_t + \begin{bmatrix}
        0 & 0 & 0 & 0\\
        0 & 0 & 0 & 0 \\
        1 & 0 & 0 & 0\\
        0 & 1 & 0 & 0\\
        0 & 0 & 0 & 0\\
        0 & 0 & 0 & 0 \\
        0 & 0 & 1 & 0\\
        0 & 0 & 0 & 1\\
    \end{bmatrix}u_t.
\end{align*}
Here, the state $x = [z_1,y_1,v_{z_1},v_{y_1}, z_2,y_2,v_{z_2},v_{y_2}]^\top$ contains the position $(z_i,y_i)$ of the center of mass of the $i^{th}$ robot and its velocity in $z$ and $y$ directions. The input $u = [a_{z_1}, a_{y_1}, a_{z_2}, a_{y_2}]^\top$ consists of the acceleration of both agents in $z$ and $y$ directions. The objective of this problem is to steer the agents from the initial point $x_S = [0, 1, 0, 0, 0, -1, 0, 0]^\top$ to the target point $x_F = [5, 3, 0, 0, 6, 2, 0, 0]^\top$ while constraining the risk of colliding with a square obstacle of length $\ell_\OO=1.8$ that moves randomly around the point $o = [3, 1.2]^\top$. Specifically, the position of the obstacle in each time-step is given by $o_t = o + \begin{bmatrix}\frac{1}{\sqrt{2}} & -\frac{1}{\sqrt{2}}\end{bmatrix}^\top w_t$,
where $w_t\in\real$ is the uncertainty. We assume that $w_t$ is defined by the $\text{Beta-binomial}(9, 10, 9)$ distribution supported on the set of nine points  
	$\setdef{-0.9 + \frac{1.8i}{8}}{i \in \until{8}}$.
Since a small number of samples are usually available in practice, we start with only $N_0=5$ samples. We assume that in each time-step of each iteration, the obstacle's position is observable, which forms the dataset of samples.
Given position $o_t$, the region of the environment occupied by the obstacle is given as
\begin{align*}
	\OO_t := \setdef{(z,y) \in \real^2}{o_t-\tfrac{\ell_\OO}{2}  \ones_2 \le [z \, \, y]^\top  \le o_t + \tfrac{\ell_\OO}{2}  \ones_2},
\end{align*}
where $\ones_2 = [1, 1]^\top$. Thus, the region $\OO_t$ can be equivalently written in form of $\OO_t = \setdef{[z \, \, y]^\top+w}{A_o[z \, \, y]^\top\le b_o}$,  where $A_o$ and $b_o$ are appropriately chosen.
According to the objective of avoiding collision in  this experiment, the constraint function $g$ is given as the minimum distance between the agents and the region $\OO_t$ occupied by the obstacle. More precisely, $g(x,w) := \max_i\{d_{\mathrm{min}} - \dist(C_ix,\OO_t)\}$, where $d_{\mathrm{min}} = 0.02$ and $\dist(C_ix,\OO_t)$ stands for the Euclidean distance of the $i^{th}$ agent from the set $\OO_t$. Specifically,
\begin{align*}
	\dist(Cx,\OO_t) = \min_{p \in \OO_t} \norm{Cx - p}.
\end{align*}
As the function $g$, and so the distributionally robust constraint~\eqref{eq:DR-RLMPC:main} are non-convex, and non-trivial to handle, we use the reformulation provided by the Proposition~\ref{prop:TV-dual} and the dual form of $g$ derived in~\cite{AN-ARH:2022}.

The stage cost is quadratic, given as ${r(x_t, u_t)=(x_F-x_t)^{\top}Q(x_F-x_t) + u_t^{\top}Ru_t}$, where ${Q=I_8}$ and ${R=\text{diag}(0.1, 0.1, 0.1, 0.1)}$. Note that $r$ satisfies the condition~\eqref{eq:st-cost}. 
The safe set $\SSs^0$ and its respective terminal cost $Q^0$, that are required for initializing the first iteration of Algorithm~\ref{ag:DR_iteration}, are generated using an open-loop controller that drives the agents to their respective targets while being far away from the obstacle, see Figure~\ref{fig:trajs}. We execute the algorithm for $20$ iterations.
 The prediction horizon is $K=5$. We use $\beta = 0.2$ as the risk-averseness coefficient and $\delta = 0$ as the right-hand side of the risk constraint.

We consider ambiguity sets defined using the total variation distance, see Section~\ref{sec:TV-reform}.
The optimization problem~\eqref{eq:DR-RLMPC:main}, considering the reformulation provided in Proposition~\ref{prop:TV-dual}, is implemented in GEKKO~\cite{LB-DH-RM-JH:2018} using APOPT solver. The implementation is available open source here~\footnote{\url{https://github.com/alirezazolanvari/DR-IRMPC}}.
\subsubsection{Results}
The obtained trajectories for different ambiguity set sizes are presented in Figure~\ref{fig:trajs}.
        \begin{figure*}
        \centering
        \begin{subfigure}[b]{0.32\linewidth}
            \centering
            \includegraphics[width=\linewidth]{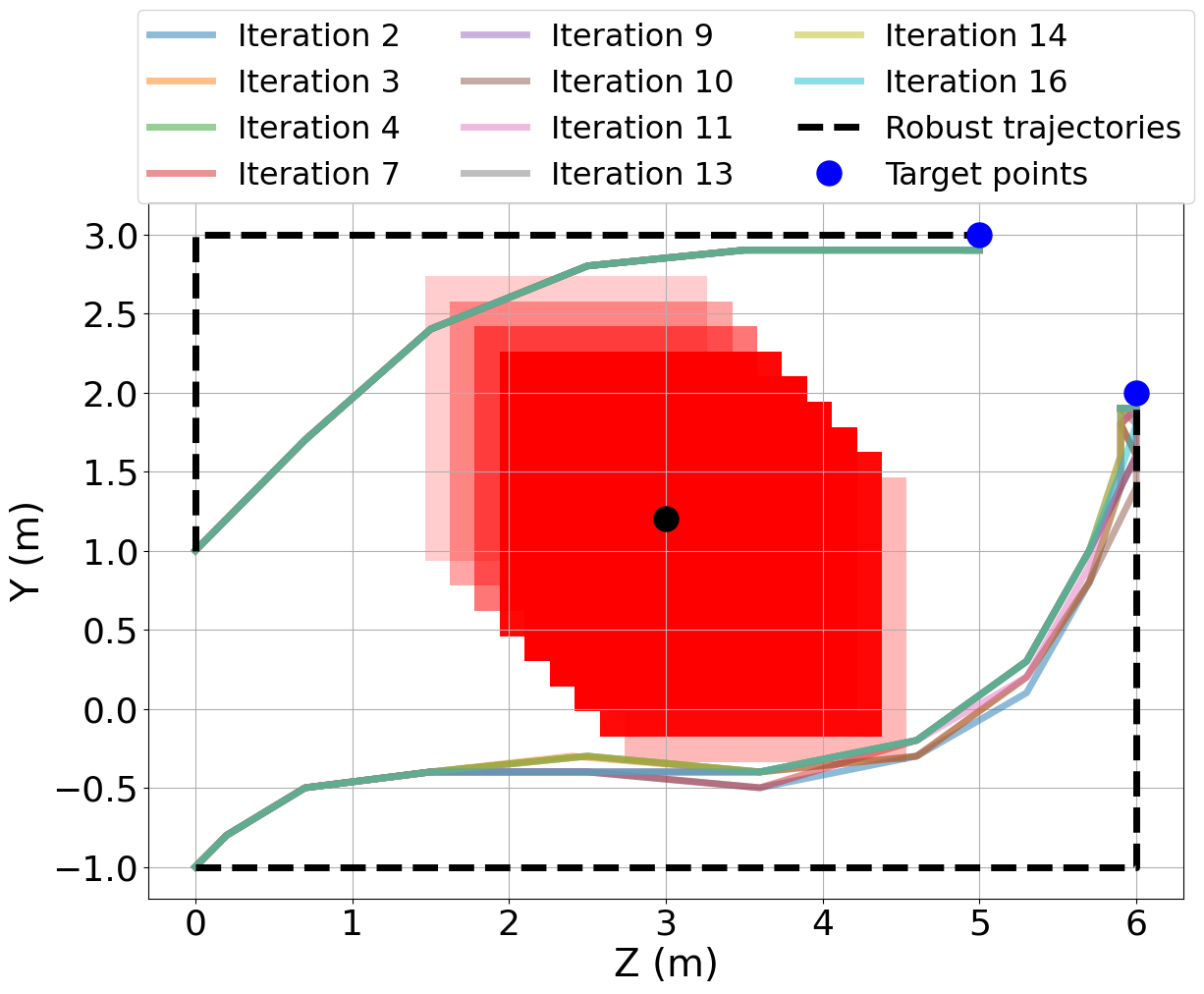}
            \caption[Network2]%
            {{\small $\theta$ = $5\times10^{-4}$}}    
            \label{fig:trajs_a}
        \end{subfigure}
        \qquad
        \begin{subfigure}[b]{0.32\linewidth}  
            \centering 
            \includegraphics[width=\linewidth]{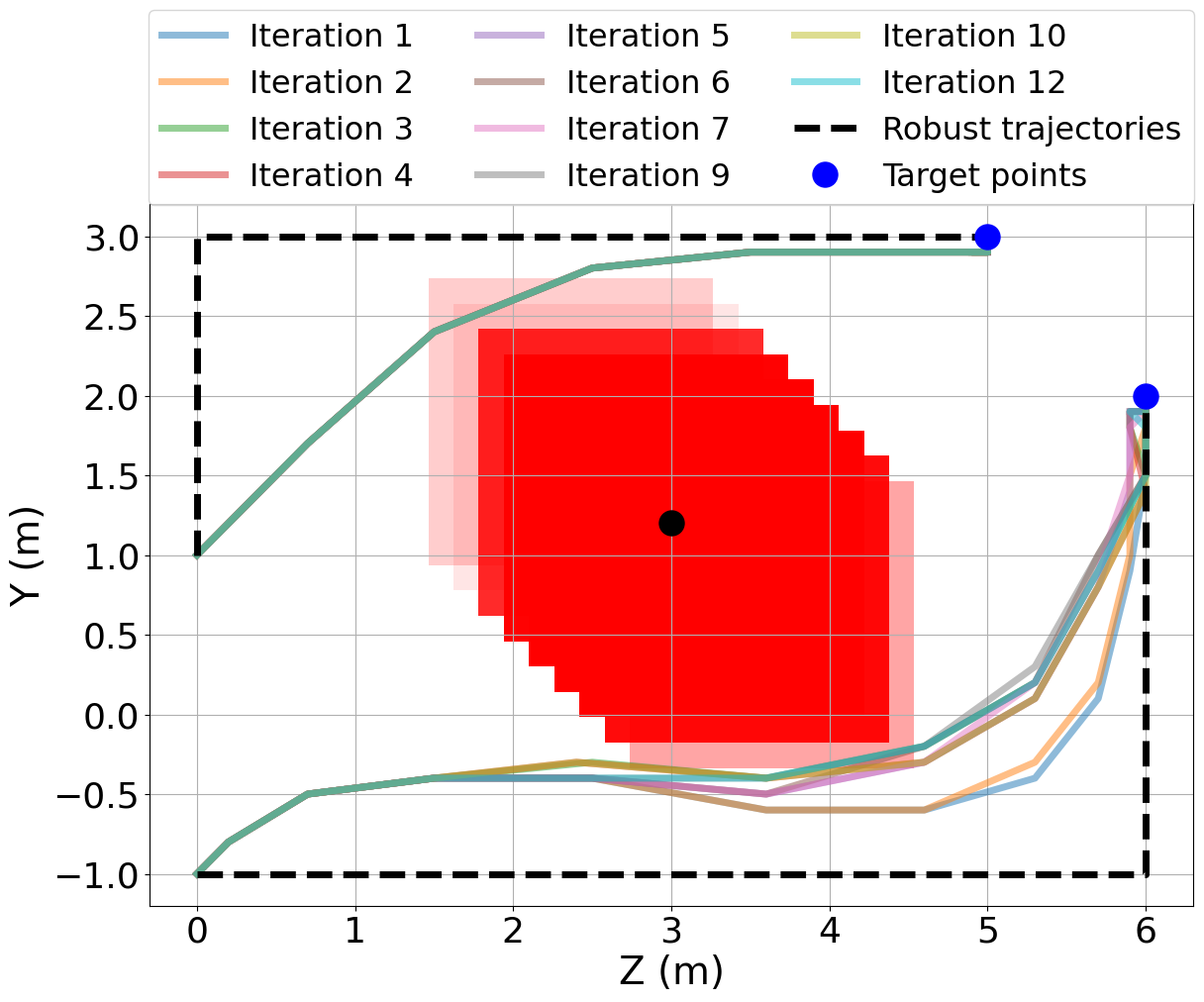}
            \caption[]%
            {{\small $\theta$ = $0.05$}}
            \label{fig:trajs_b}
        \end{subfigure}
        \vskip\baselineskip
        \begin{subfigure}[b]{0.32\linewidth}   
            \centering 
            \includegraphics[width=\linewidth]{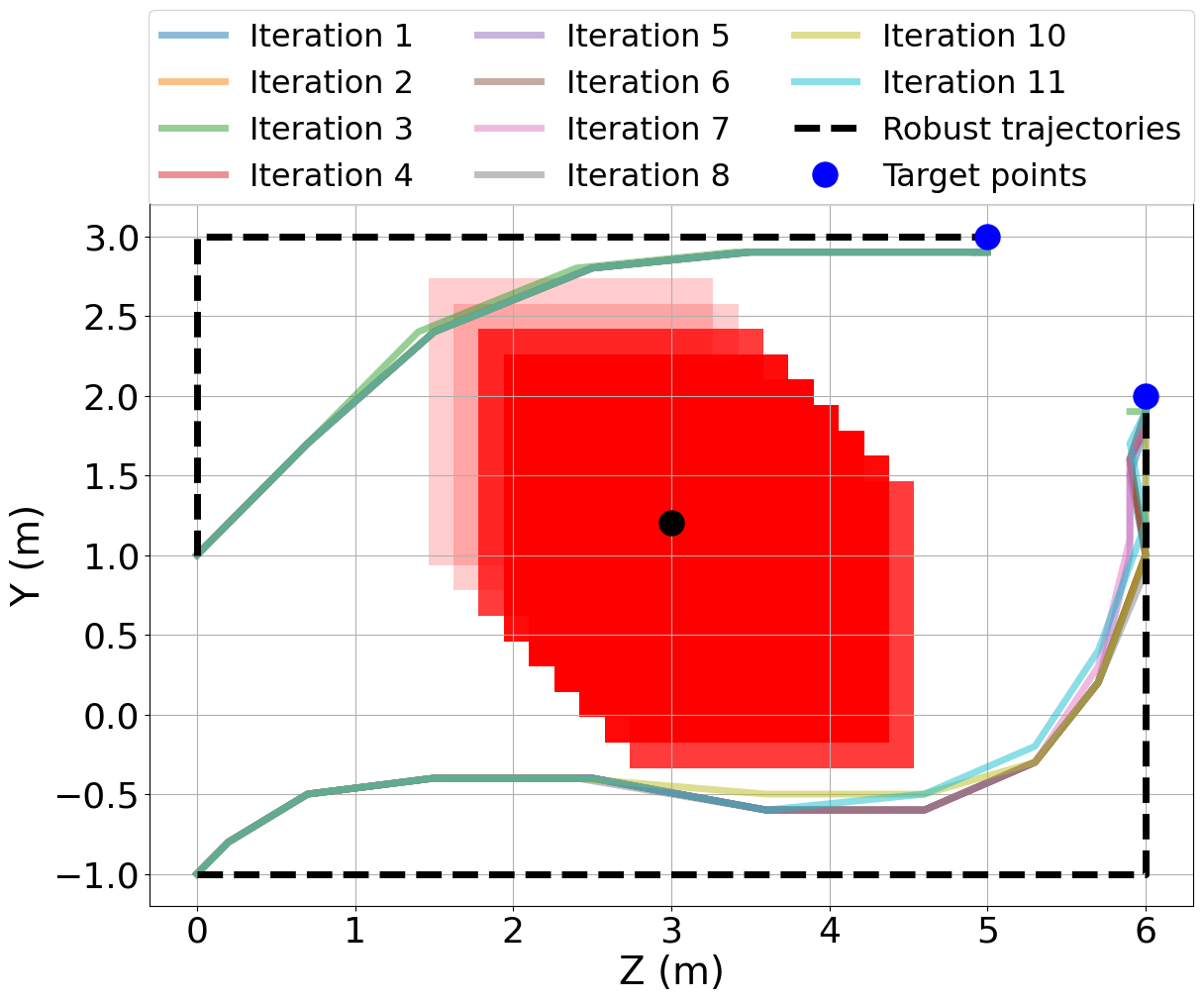}
            \caption[]%
            {{\small $\theta$ = $0.1$}}    
            \label{fig:trajs_c}
        \end{subfigure}
        \qquad
        \begin{subfigure}[b]{0.32\linewidth}   
            \centering 
            \includegraphics[width=\linewidth]{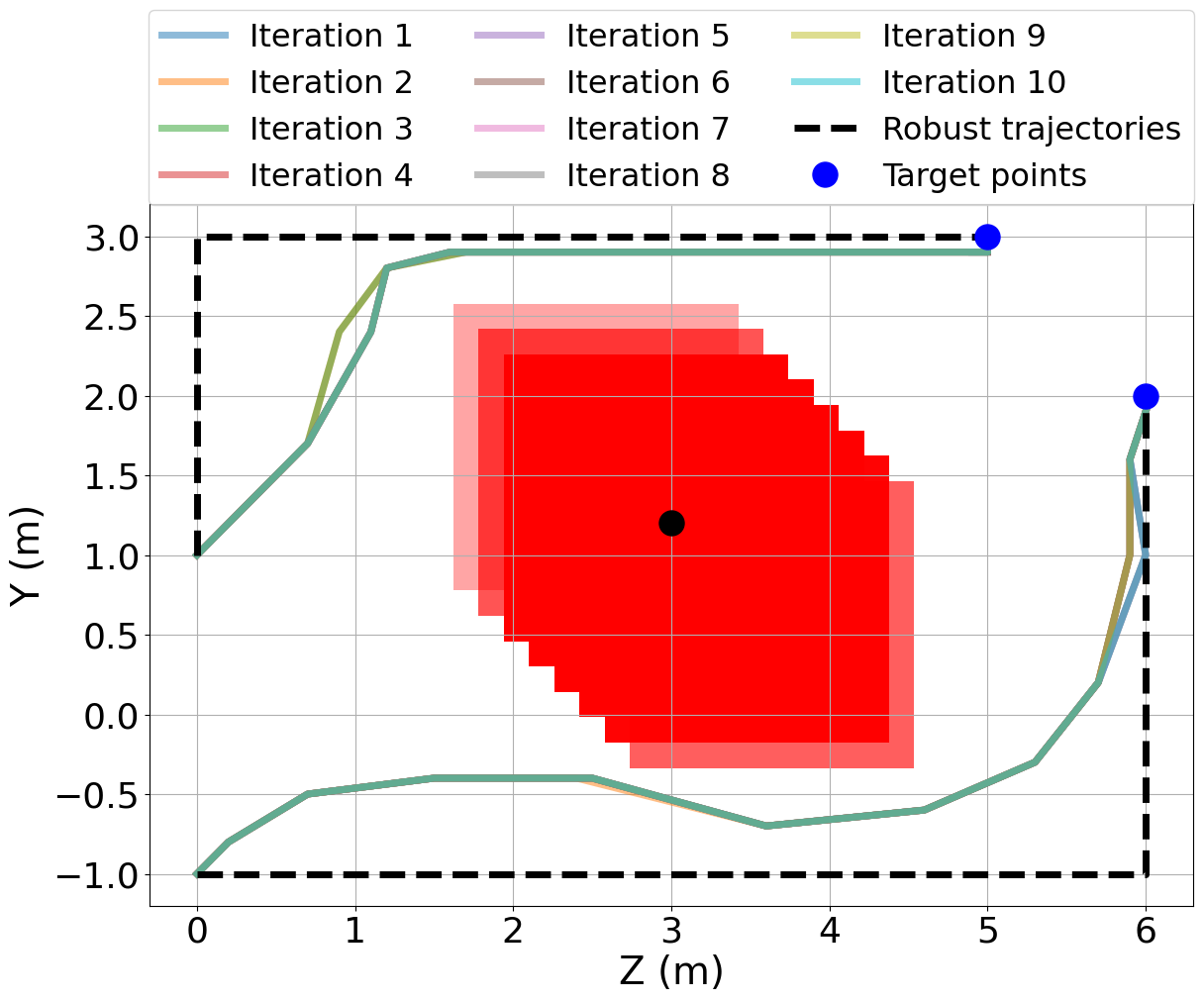}
            \caption[]%
            {{\small $\theta$ = $0.2$}}    
            \label{fig:trajs_d}
        \end{subfigure}
        \caption{\footnotesize Plots illustrating the application of the DR-RC-Iterative-MPC procedure for the task of navigating two mobile robots in an environment with randomly-moving obstacle (see Section~\ref{sec:sims} for details). We consider four different radii for the ambiguity set and for each case, the radius does not change over the iterations. The robust trajectories (dashed black lines) are same for all cases. Each realization of the obstacle is plotted with a shaded red square. Here, only the first ten trajectories in which both agents reach their targets are presented. As observed, the trajectories become more conservative as the radius of the ambiguity set increases.} 
        \vspace*{-3ex}
        \label{fig:trajs}
    \end{figure*}
    In the first iteration, due to the constraint $x_{K}\in\Pis{\SSs^0}$, the system follows a similar trajectory for $\theta \in \{0.05, 0.1, 0.2\}$. However, differences become more apparent as iterations progress. For small ambiguity sets, the trajectories are closer to the obstacle. For larger ones, the algorithm becomes more conservative to the extent that for $\theta = 0.2$ the agent stops exploring and is only concerned about safety.  There is a noteworthy observation in Figure~\ref{fig:trajs_d}, that is, trajectories of the upper agent get closer to the obstacle in the first few iterations but as more data is collected, the safe set gets refined in 
    later iterations and the robot deviates from the obstacle more strongly. Finally in Figure~\ref{fig:collision} we underline the impact of the size of the ambiguity set on cost-performance and safety. As shown in the figure, smaller ambiguity sets provide cost-efficient trajectories while they also increase the probability of colliding with the obstacle. 
    As a result, if an appropriate value is chosen for the radius of the ambiguity set, the algorithm is able to provide an acceptable level of safety even using small number of samples. 
\begin{figure}
    \centering
    \includegraphics[width=0.75\linewidth]{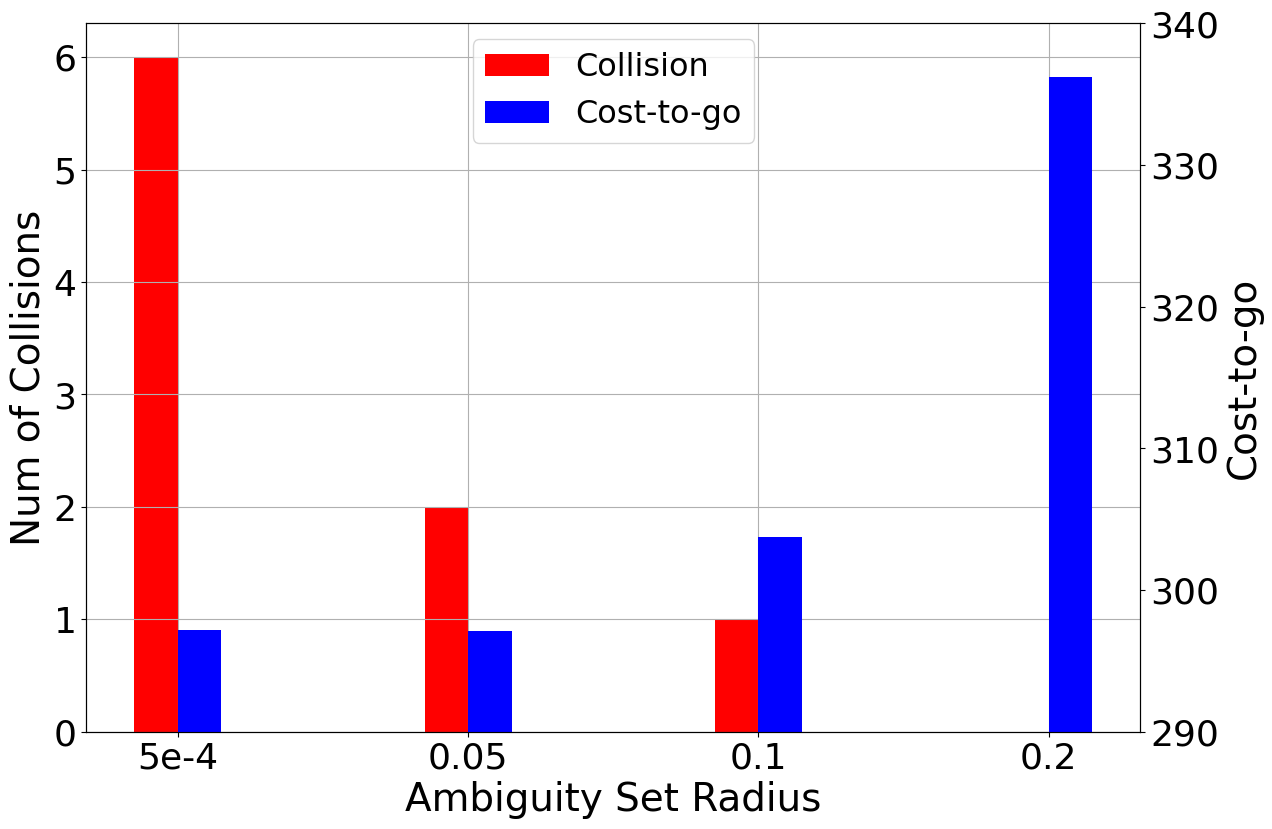}
    \caption{\footnotesize The effect of the size of the ambiguity set on safety and performance. The red block represents the number of iterations (out of $20$) in which the trajectory collides with the obstacle at least once. The blue block depicts the iteration cost of the collision-free iteration that has the highest index.}
    \vspace*{-3ex}
    \label{fig:collision}
\end{figure}

\section{Conclusions}

We considered a risk-constrained infinite-horizon optimal control problem and designed an iterative MPC-based scheme to solve it. Our procedure approximated the risk constraints using their data-driven distributionally robust counterparts. Each iteration in our method generated a trajectory that is provably safe and that converges to the equilibrium asymptotically. Lastly, we implemented our algorithm to find a risk-averse path for  mobile robots that are in an environment with uncertain obstacle. Future work will explore the convergence of the iterations to the optimal solution of the risk-constrained optimal control problem. We wish to implement this method for a multi-agent setup in a distributed manner, where different agents collect different samples of the uncertainty. Finally, with online implementation as goal, we plan to reduce the computational effort of the method by approximating the DR constraint in an effective manner. %

\bibliographystyle{ieeetr}

\appendix
\renewcommand{\theequation}{A.\arabic{equation}}
\renewcommand{\thetheorem}{A.\arabic{theorem}}
\renewcommand{\theproposition}{A.\arabic{proposition}}

The following result on attractivity of the equilibrium point of a discrete-time system aids us in showing convergence of our iterative $\drmpc$ scheme. The proof follows standard Lyapunov arguments but the exact result is not available in the literature. We provide the proof here for completeness.

\begin{proposition}\label{prop:xbar_conv}\longthmtitle{Attractivity of discrete-time system}
	Consider the system
	\begin{align}\label{eq:gen-sys}
		x_{t+1} = f(x_t), \quad x_0 \in \XX \subset \real^n,
	\end{align}
	where $\map{f}{\XX}{\XX}$ and $\XX$ is a compact set. Given a point $x^* \in \XX$, let the function $\map{V}{\XX}{\realnonnegative}$ satisfy 
	\begin{align}
		V(x^*) = 0, \quad V(x) > 0 \quad \forall x \in \XX \setminus \{x^*\}.
	\end{align}
	Assume there exists a continuous function $\map{\phi}{\XX}{\realnonnegative}$ such that $\phi(x^*) = 0$, $\phi(x) > 0$ for all $x \in \XX\setminus \{x^*\}$, and 
	\begin{align}\label{eq:V-lyap}
		V(f(x))-V(x)\leq-\phi(x) \quad \text{ for all } x \in \XX.
	\end{align}
	Then, any trajectory $\{x_t\}$ of~\eqref{eq:gen-sys} satisfies $\lim_{t \to \infty} x_t = x^*$. 
\end{proposition}
\begin{proof}
	By contradiction, assume that there exists a trajectory $\{x_t\}_{t=0}^\infty$ of~\eqref{eq:gen-sys} such that $\lim_{t \to \infty} x_t \not = x^*$. Using this fact and the compactness of $\XX$, there exists a subsequence of $\{x_t\}_{t=0}^\infty$, denoted as $\{x_{t_k}\}_{k=0}^\infty$, such that $\lim_{k\rightarrow\infty}x_{t_k}=\bar{x}$ and $\bar{x}\neq x^*$. 
Let $\epsilon>0$ be such that $\NN_\epsilon:=\setdef{x \in \XX}{\norm{x-\bar{x}}\leq\epsilon}$ does not contain $x^*$. Let $\bar{\phi}:=\min\limits_{x\in\NN_\epsilon}\phi(x) > 0$. This is well defined as $\NN_\epsilon$ is compact and $\phi$ is continuous. Since $x_{t_k} \to \bar{x}$, there exists a $K$ such that $x_{t_k} \in \NN_\eps$ for all $k \ge K$. Using~\eqref{eq:V-lyap} and the definition of $\bar{\phi}$,  we have
\begin{align}
	V(x_{t_{k+1}})  \le V(x_{t_k}) - \phi(x_{t_k} ) 
 \le V(x_{t_k}) - \bar{\phi} \label{eq:lyap-inf-dec},
\end{align}
for all $k \ge K$. The sequence $\{V(x_t)\}_{t=1}^\infty$ is non-increasing due to~\eqref{eq:V-lyap}. This fact along with~\eqref{eq:lyap-inf-dec} and the lower bound on $V$ yields a contradiction. This completes the proof. 
\end{proof}

\end{document}